\documentclass[a4paper,11pt]{amsart}

\usepackage{amsmath}
\usepackage{amssymb}
\usepackage{amsthm}
\usepackage[pdftex]{graphicx}
\usepackage{amscd}
\usepackage{comment}
\usepackage{setspace}
\usepackage{url}                 

\begin{document}

\numberwithin{equation}{section}

\theoremstyle{definition}
\newtheorem{theorem}{Theorem}[section]
\newtheorem{proposition}[theorem]{Proposition}
\newtheorem{lemma}[theorem]{Lemma}
\newtheorem{corollary}[theorem]{Corollary}
\newtheorem{remark}[theorem]{Remark}
\newtheorem{definition}[theorem]{Definition}
\newtheorem{example}[theorem]{Example}
\newtheorem{problem}[theorem]{Problem}
\newtheorem{convention}[theorem]{Convention}
\newtheorem{conjecture}[theorem]{Conjecture}

\title[Acylindrical hyperbolicity of Artin groups]
{Acylindrical hyperbolicity of Artin groups associated with graphs that are not cones}

\author[M. Kato]{Motoko Kato}
\thanks{The first author is supported by JSPS KAKENHI Grant Numbers 19K23406, 20K14311, 
and JST, ACT-X Grant Number JPMJAX200A, Japan.}
\address[M. Kato]{Faculty of Education, University of the Ryukyus, 1 Sembaru, 
Nishihara-Cho, Nakagami-Gun, Okinawa 903-0213, Japan}
\email{katom@edu.u-ryukyu.ac.jp}

\author[S. Oguni]{Shin-ichi Oguni}
\thanks{The second author is supported by JSPS KAKENHI Grant Number 20K03590.}
\address[S. Oguni]{Department of Mathematics, Faculty of Science, Ehime University, 2-5 Bunkyo-cho, Matsuyama, Ehime 790-8577, Japan}
\email{oguni.shinichi.mb@ehime-u.ac.jp}


\maketitle

\begin{abstract}
Charney and Morris-Wright showed acylindrical hyperbolicity of 
Artin groups of infinite type associated with graphs that are not joins, 
by studying clique-cube complexes and the actions on them. 
In this paper, by developing their study and formulating some additional discussion, 
we demonstrate that acylindrical hyperbolicity holds for more general Artin groups. 
Indeed, we are able to treat Artin groups of infinite type associated with graphs that are not cones. 
\end{abstract}

\section{Introduction}

Artin groups, also called Artin-Tits groups, have been widely studied 
since their introduction by Tits \cite{Tits}. 
In particular, Artin groups are important examples in geometric group theory. 
For various nonpositively curved or negatively curved properties on discrete groups, 
Artin groups are interesting targets. 
In this paper, we consider acylindrical hyperbolicity of Artin groups. 

Let $\Gamma$ be a finite simple graph 
with the vertex set $V=V(\Gamma)$ and the edge set $E=E(\Gamma)$. 
Each edge $e$ has two end vertices, which we denote as $s_e$ and $t_e$. 
We suppose that any edge $e$ is labeled by an integer $\mu(e)\ge 2$. 
The \textit{Artin group} $A_{\Gamma}$ associated with $\Gamma$ is defined by the following presentation 
	\begin{align}\label{pres_standard_eq}
	A_{\Gamma}=\langle V(\Gamma)\mid \underbrace{s_e t_e s_e t_e \cdots}_{\text{length $\mu(e)$}}	
	=\underbrace{t_e s_e t_e s_e\cdots}_{\text{length $\mu(e)$}} \quad \text{for all $e\in E(\Gamma)$ } 	\rangle.
	\end{align}
Free abelian groups, free groups, and braid groups are typical examples of Artin groups.
Adding the relation $v^2 = 1$ for all $v\in V(\Gamma)$ to (\ref{pres_standard_eq}) 
produces the associated \textit{Coxeter group} $W_\Gamma$. 
In terms of the properties of $W_{\Gamma}$, 
we can define several important classes of Artin groups. 
The Artin group $A_\Gamma$ is said to be \textit{of finite type} 
if $W_\Gamma$ is finite. Otherwise, it is said to be \textit{of infinite type}. 
The Artin group $A_\Gamma$ is said to be \textit{irreducible} 
if $W_\Gamma$ is \textit{irreducible}, that is, 
the defining graph $\Gamma$ cannot be decomposed as a join of two subgraphs 
such that all edges between them are labeled by $2$. 
It is well-known that an infinite Coxeter group $W_\Gamma$ is irreducible if and only if 
$W_\Gamma$ cannot be directly decomposed into two nontrivial subgroups (\cite{MR2240393} and \cite{MR2333366}). 
However, it is unclear 
whether $A_\Gamma$ is irreducible if and only if $A_\Gamma$ cannot be directly decomposed 
into two nontrivial subgroups. 
In general, Coxeter groups are well understood, 
but many basic questions for Artin groups remain open 
(refer to \cite{Charney2008PROBLEMSRT} and \cite{MR3203644}). 

We consider nonpositively curved or negatively curved properties on Artin groups. 
The following is one of the most important open problems 
\cite[Problem 4]{Charney2008PROBLEMSRT}. 
\begin{problem}\label{CAT(0)_conj}
Which Artin groups are \textit{CAT(0) groups}, that is, groups acting geometrically on CAT(0) spaces?
\end{problem}
\noindent 
Here, \textit{CAT(0) spaces} are geodesic spaces in which every geodesic triangle is not fatter than 
the comparison triangle in the Euclidean plane (see \cite{BH} for the precise definition).
A group action is said to be \textit{geometric} if the action is proper, cocompact, and isometric. 
In recent studies on geometric group theory, various properties besides the CAT(0) property 
have been actively investigated, such as systolic property and the Helly property 
(see, for example, \cite{HO20}, \cite{HO21}). 

In this paper, we consider the following problem \cite[Conjecture B]{haettel2019xxl}. 
\begin{problem}\label{acyl_conj}
Are irreducible Artin groups of infinite type acylindrically hyperbolic?
\end{problem}
\noindent 
The definition of acylindrical hyperbolicity is given in Section \ref{recall}.
There are many applications of acylindrical hyperbolicity 
(see, for example, \cite{MR3589159}, \cite{Osin}, and \cite{Osin2}). 
\begin{remark} 
\begin{enumerate}
\item Reducible Artin groups can be directly decomposed into two infinite subgroups. 
However, acylindrical hyperbolic groups cannot be directly decomposed into two infinite subgroups 
\cite[Corollary 7.3]{Osin}. 
Hence, such Artin groups are not acylindrically hyperbolic. 
\item Irreducible Artin groups of finite type have infinite cyclic centers 
\cite{MR323910}, \cite{MR422673}. 
Because acylindrical hyperbolic groups do not permit infinite centers \cite[Corollary 7.3]{Osin}, 
such Artin groups are not acylindrically hyperbolic. 
We remark that the central quotients for irreducible Artin groups of finite type are 
acylindrically hyperbolic 
(see \cite{MR1914565}, \cite{MR2390326}, \cite{MR2367021} for braid groups, 
and \cite{MR3719080} for the general case). 
\end{enumerate}
\end{remark}

Many affirmative partial answers for Problem \ref{acyl_conj} are known. 
Indeed, the following irreducible Artin groups of infinite type are known to be 
acylindrically hyperbolic: 
\begin{itemize}
\item Right-angled Artin groups \cite{MR2827012}, \cite{MR3192368}; 
\item Two-dimensional Artin groups such that the associated Coxeter groups are hyperbolic \cite{alex2019acylindrical}; 
\item Artin groups of XXL-type \cite{haettel2019xxl}; 
\item Artin groups of type FC such that the defining graphs have diameter greater than $2$ \cite{MR3966610}; 
\item Artin groups that are known to be  CAT(0) groups according to the result of Brady and McCammond \cite{MR1770639}, \cite{KO}; 
\item Euclidean Artin groups \cite{Calvez}. 
\end{itemize}
Actually, except for Euclidean Artin groups, 
all of these Artin groups are regarded as special cases of 
the following irreducible Artin groups of infinite type, 
which are known to be acylindrically hyperbolic: 
\begin{itemize}
\item Artin groups associated with graphs that are not joins \cite{Charney}; 
\item \textit{Two-dimensional} Artin groups, that is, 
Artin groups such that every triangle with three vertices $v_1,v_2,v_3$ 
of the defining graphs satisfies $\frac{1}{\mu((v_1,v_2))} + \frac{1}{\mu((v_2,v_3))}  + \frac{1}{\mu((v_3,v_1))} \le 1$ 
\cite{Vaskou}. 
\end{itemize}

Charney and Morris-Wright \cite{Charney} showed acylindrical hyperbolicity of 
Artin groups of infinite type associated with graphs that are not joins, 
by studying clique-cube complexes, 
which are CAT(0) cube complexes, 
and the isometric actions on them. 
In fact, they constructed a WPD contracting element of such an Artin group 
with respect to the isometric action on the clique-cube complex. 
In this paper, we generalize this result 
by developing their study and formulating some additional discussion. 
Our main theorem can be stated as the follows. 
\begin{theorem}\label{main}
Let $A_\Gamma$ be an Artin group associated with 
$\Gamma$, where $\Gamma$ has at least three vertices. 
Suppose that $\Gamma$ is not a cone. 
Then, the following are equivalent:
\begin{enumerate}
\item $A_\Gamma$ is irreducible, that is, 
$\Gamma$ cannot be decomposed as a join of two subgraphs such that all edges between them are labeled by $2$; 
\item $A_\Gamma$ has a WPD contracting element with respect to the isometric action on the clique-cube complex;
\item $A_\Gamma$ is acylindrically hyperbolic;
\item $A_\Gamma$ is directly indecomposable, that is, it cannot be decomposed as a direct product 
of two nontrivial subgroups. 
\end{enumerate}
\end{theorem}

\begin{remark}
When an Artin group $A_\Gamma$ is irreducible 
and the defining graph $\Gamma$ 
is not a cone, 
the center $Z(A_\Gamma)$ is known to be trivial. 
This fact is shown in \cite{Charney}. We present 
an alternative proof based on Theorem \ref{main} (see Remark \ref{mainremark}). 
\end{remark}

From Theorem \ref{main}, we find that 
many irreducible Artin groups of infinite type are acylindrically hyperbolic, e.g., 
the Artin groups associated with the defining graphs in Figure \ref{new_ex_fig}. 
\begin{figure}
\begin{center}
\includegraphics[width=12cm,pagebox=cropbox,clip]{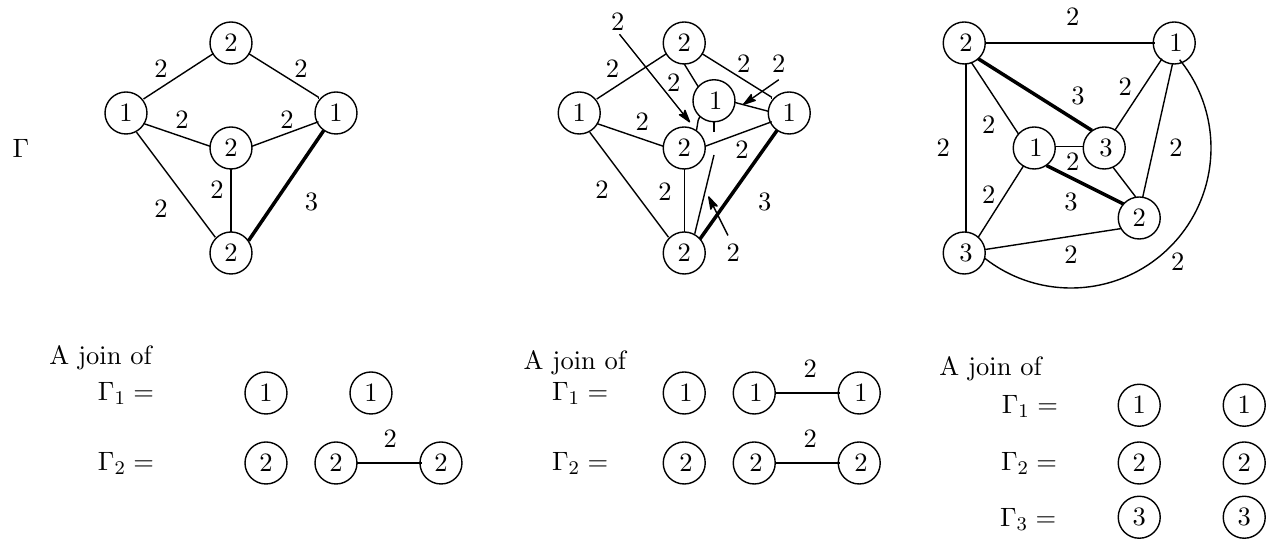} 
\caption{Defining graphs $\Gamma$ of new examples $A_{\Gamma}$.}\label{new_ex_fig}
\end{center}
\end{figure}

The remainder of this paper is organized as follows. 
Section \ref{recall} contains some preliminaries regarding acylindrically hyperbolic groups, WPD contracting elements, 
and CAT(0)-cube complexes. 
Section \ref{artin} presents preliminaries on defining graphs of Artin groups and joins of graphs. 
In Section \ref{clique}, we treat clique-cube complexes and the actions on them by Artin groups 
following \cite{Charney}. 
In Section \ref{local_geom}, we study the local geometry of clique-cube complexes. 
Section \ref{main_proof} gives a proof of Theorem \ref{main}. 
Our main task is to construct a candidate WPD contracting element 
and show that it really is a WPD contracting element. 

\section{Acylindrical hyperbolicity, WPD contracting elements, and CAT(0)-cube complexes}\label{recall}
In this section, we collect some definitions and properties related to acylindrical hyperbolicity, 
WPD contracting elements, and CAT(0)-cube complexes that will be used later in the paper. 
See \cite{Genevois} and the references therein for details. 

First, we recall the definition of acylindrical hyperbolicity (see \cite{Osin}). 
\begin{definition}\label{acylindrically hyperbolic}
A group $G$ is \textit{acylindrically hyperbolic} if it admits an isometric action 
on a hyperbolic space $Y$ that is \textit{non-elementary}  (i.e., with an infinite limit set) 
and \textit{acylindrical} (i.e., for every $D\ge 0$, there exist some $R, N \ge 0$ such that, 
for all $y_1, y_2\in Y$, $d_Y(y_1, y_2)\ge R$ implies 
$\#\{g\in G |\ d_Y(y_1, g(y_1)), d_Y(y_2, g(y_2)) \le D\}\le N$).
\end{definition}

Next, we recall the definition of a WPD contracting element. 
\begin{definition}\label{WPD}
Let a group $G$ act isometrically on a metric space $X$. For $\gamma\in G$, 
we say that: 
\begin{itemize}
\item $\gamma$ is \textit{WPD} if, for every $D \ge 0$ and $x \in X$, there exists some $M\ge 1$ 
such that $\#\{g\in G |\ d_X(x, g(x)), d_X(\gamma^M(x), g\gamma^M(x)) \le D\}<\infty$; 
\item $\gamma$ is \textit{contracting} if $\gamma$ is \textit{loxodromic}, that is, 
there exists $x_0\in X$ such that $\mathbb Z \to X; n \mapsto \gamma^n(x_0)$ is a quasi-isometry onto the image 
$\gamma^\mathbb{Z}x_0:=\{\gamma^n(x_0) | n \in \mathbb Z\}$, 
and $\gamma^\mathbb{Z}x_0$ is \textit{contracting}, that is, there exists $B\ge 0$ such that 
the diameter of the nearest-point projection of any ball that is disjoint 
from $\gamma^\mathbb{Z}x_0$ onto $\gamma^\mathbb{Z}x_0$ is bounded by $B$. 
\end{itemize}
\end{definition}

The following is a consequence of \cite{BBF}. 
\begin{theorem}\label{BBF}
Let a group $G$ act isometrically on a geodesic metric space $X$. 
Suppose that $G$ is not virtually cyclic. 
If there exists a WPD contracting element $\gamma\in G$, then $G$ is acylindrically hyperbolic. 
\end{theorem}

CAT(0) cube complexes are considered as generalized trees in higher dimensions. 
The following is a precise definition (see \cite[p.111]{BH}). 

\begin{definition}\label{CAT(0)}
A \textit{cube complex} is a $CW$ complex constructed by gluing together cubes of 
arbitrary (finite) dimension by isometries along their faces. 
Furthermore, the cube complex is \textit{nonpositively curved} if the link of any of its vertices is 
a \textit{simplicial flag} complex (i.e., $n+1$ vertices span an $n$-simplex if and only if they are pairwise adjacent), 
and \textit{CAT(0)} if it is nonpositively curved and simply connected. 
\end{definition}

\begin{definition}
Let $X$ be a CAT(0) cube complex. 
We define an equivalence relation for the edges of $X$ as 
the transitive closure of the relation identifying two parallel edges of a square. 
For an equivalence class, a \textit{hyperplane} is defined as 
the union of the midcubes transverse to the edges belonging to the equivalence class. 
Then, for any edge belonging to the equivalence class, 
the hyperplane is said to be \textit{dual to} the edge.

For a hyperplane $J$, 
we denote 
the union of the cubes intersecting $J$ as $N(J)$, that is, 
the smallest subcomplex of $X$ containing $J$. 
We denote 
the union of the cubes not intersecting $J$ as $X\!\setminus\!\!\setminus J$, 
that is, the largest subcomplex of $X$ not intersecting $J$. 
\end{definition}

See \cite{Sageev} for the following. 
\begin{theorem}\label{separate}
Let $X$ be a CAT(0) cube complex and $J$ be a hyperplane. 
Then, $X\!\setminus\!\!\setminus J$ has exactly two connected components. 
\end{theorem}
\noindent
The two connected components of $X\!\setminus\!\!\setminus J$ 
are often denoted as $J^+$ and $J^-$, respectively. 

For convenience, we prepare the following for the proof of Theorem \ref{main}. 
\begin{definition}\label{separatingdef}
Let $X$ be a CAT(0) cube complex. 
For two vertices $x$ and $x'$ in $X$, we call a sequence of hyperplanes 
$P_1,\ldots,P_M$ a \textit{sequence of separating hyperplanes from $x$ to $x'$} 
if the sequence satisfies 
$$x \in P_1^-, P_1^+\supsetneq P_2^+\supsetneq\cdots 
\supsetneq P_{M-1}^+\supsetneq P_M^+\ni x'$$ 
for some connected components $P_i^+$ of $X\!\setminus\!\!\setminus P_i$ 
for all $i\in \{1,\ldots,M\}$. 

For two hyperplanes $J$ and $J'$ in $X$, we call a sequence of hyperplanes 
$P_1,\ldots,P_M$ a \textit{sequence of separating hyperplanes from $J$ to $J'$} 
if the sequence satisfies 
$$J^+ \supsetneq P_1^+\supsetneq P_2^+\supsetneq\cdots 
\supsetneq P_{M-1}^+\supsetneq P_M^+ \supsetneq J'^+$$ 
for some connected components $J^+$ of $X\!\setminus\!\!\setminus J$, 
$J'^+$ of $X\!\setminus\!\!\setminus J'$, 
and $P_i^+$ of $X\!\setminus\!\!\setminus P_i$ 
for all $i\in \{1,\ldots,M\}$. 
\end{definition}
\begin{remark}\label{separatingrem}
When $P_1,\ldots,P_M$ is a sequence of separating hyperplanes from $J$ to $J'$, 
for two vertices $x\in J^-\cup N(J)$ and $x'\in J'^+\cup N(J')$, 
$P_1,\ldots,P_M$ is a sequence of separating hyperplanes from $x$ to $x'$. 
\end{remark}

The following is part of \cite[Theorem 3.3]{Genevois}, 
and is used in the proof of Theorem \ref{main}. 
\begin{theorem}\label{criterion}
Let a group $G$ act isometrically on a CAT(0) cube complex $X$. 
Then, $\gamma \in G$ is a WPD contracting element if there exist 
two hyperplanes $J$ and $J'$ satisfying the following: 
\begin{enumerate}
\item[(i)] $J$ and $J'$ are strongly separated, that is, 
no hyperplane can intersect both $J$ and $J'$; 
\item[(ii)] $\gamma$ skewers $J$ and $J'$, that is, we have connected components $J^+$ of $X\!\setminus\!\!\setminus J$ and 
$J'^+$ of $X\!\setminus\!\!\setminus J'$ such that 
$\gamma^n (J^+) \subsetneq J'^+ \subsetneq J^+$ for some $n\in\mathbb N$;
\item[(iii)] $\mathrm{stab}(J) \cap \mathrm{stab} (J')$ is finite, where 
$\mathrm{stab}(J)=\{g\in G\  | \ g(J)=J\}$ and $\mathrm{stab} (J')=\{g\in G\  | \ g(J')=J'\}$. 
\end{enumerate}
\end{theorem}

\section{Defining graphs of Artin groups and joins}\label{artin}
\subsection{Defining graphs of Artin groups}\label{defining_graph}
We now present a precise description of 
the defining graph of an Artin group and introduce some related graphs. 

Let $V$ be a finite set. 
Denote the diagonal set as $\mathrm{diag}(V\times V):=\{(v,w)\in V\times V |\ v=w\}$. 
We consider the involution on the off-diagonal set 
$$\iota:V\times V\setminus\mathrm{diag}(V\times V)\ni (v,w)\mapsto (w,v)\in V\times V\setminus\mathrm{diag}(V\times V).$$ 
Any $e\in V\times V\setminus\mathrm{diag}(V\times V)$ is often presented as $(s_e,t_e)$. 
Then, for any $e\in V\times V\setminus\mathrm{diag}(V\times V)$, 
we have $s_{\iota(e)}=t_e$ and $t_{\iota(e)}=s_e$. 
We take a symmetric map 
$$\tilde{\mu}:V\times V\setminus\mathrm{diag}(V\times V) \to \mathbb{Z}_{\ge 2}\cup \{\infty\}.$$ 
Here, `\textit{symmetric}' means that $\tilde{\mu}\circ\iota=\tilde{\mu}$ is satisfied. 
Set $E_m:=\tilde{\mu}^{-1}(m)$ for any $m\in \mathbb{Z}_{\ge 2}\cup \{\infty\}$. 
Then, we have 
$$V\times V\setminus\mathrm{diag}(V\times V)=\bigsqcup_{m\in \mathbb{Z}_{\ge 2}\cup \{\infty\}}E_m.$$
We now have a finite simple labeled graph $\Gamma$ 
with the vertex set $V(\Gamma)=V$, 
the edge set $E(\Gamma)=\bigsqcup_{m\in \mathbb{Z}_\ge 2}E_m$, 
and the labeling $\mu:=\tilde{\mu}|_{E(\Gamma)}$. 
The Artin group $A_\Gamma$ associated with $\Gamma$ 
is then defined by the presentation (\ref{pres_standard_eq}) 
and $\Gamma$ is called the \textit{defining graph} of $A_\Gamma$. 

For convenience, we define 
two other finite simple graphs $\Gamma^c$ and $\Gamma^t$ as follows. 
$\Gamma^c$ is the finite simple graph 
with the vertex set $V(\Gamma^c)=V$ 
and the edge set $E(\Gamma^c)=E_{\infty}$. 
This is the so-called \textit{complement graph} of $\Gamma$. 
$\Gamma^t$ is the finite simple graph with the vertex set $V(\Gamma^t)=V$ 
and the edge set $E(\Gamma^t)=\bigsqcup_{m\in \mathbb{Z}_{\ge 3}\cup\{\infty\}}E_m$. 
See Figures \ref{new_ex_fig} and \ref{traditional_fig}. 

\begin{figure}
\begin{center}
\includegraphics[width=12cm,pagebox=cropbox,clip]{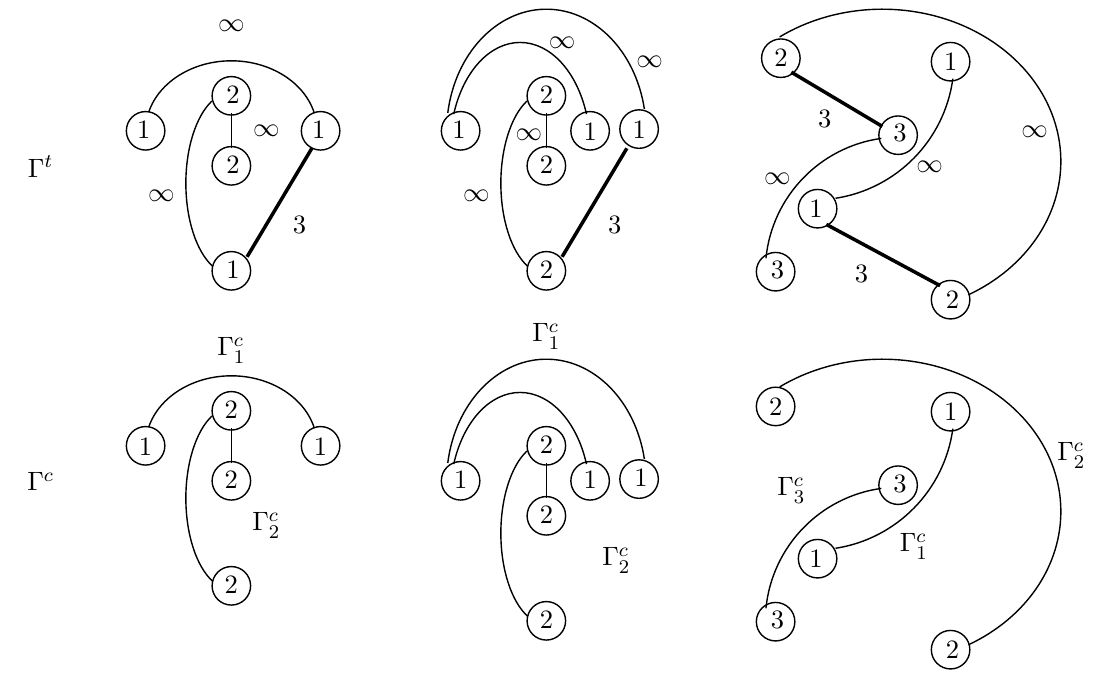} 
\caption{$\Gamma^t$ and $\Gamma^c$ with respect to $\Gamma$ in Figure \ref{new_ex_fig}.}\label{traditional_fig}
\end{center}
\end{figure}

\begin{remark}
In research related to Coxeter groups and some traditional treatments of Artin groups, 
$\Gamma^t$ is used with the label $\tilde{\mu}|_{E(\Gamma^t)}$. 
We mainly use $\Gamma$ 
in accordance with many recent studies on Artin groups. 
We only use $\Gamma^t$ as an aid in this paper. 
\end{remark}

\subsection{Joins}
\begin{definition}
Let $\Lambda\neq \emptyset$ be an index set. 
The \textit{join} $\ast_{\alpha\in \Lambda}\Gamma_\alpha$ 
of simple graphs $\Gamma_\alpha$ ($\alpha\in \Lambda$) is defined as a simple graph with 
the vertex set $$V(\ast_{\alpha\in \Lambda} \Gamma_\alpha):=\bigsqcup_{\alpha\in \Lambda}  V(\Gamma_\alpha)$$ 
and the edge set $$E(\ast_{\alpha\in \Lambda} \Gamma_\alpha):=\bigsqcup_{\alpha\in \Lambda} E(\Gamma_\alpha)\sqcup 
\bigsqcup_{\alpha,\beta\in \Lambda, \alpha\neq \beta}\{(v_\alpha, v_\beta)\mid v_\alpha\in V(\Gamma_\alpha), v_\beta\in V(\Gamma_\beta)\}.$$

A simple graph $\Gamma$ is said to be \textit{decomposable} (as a join) 
if there exist an index set $\Lambda$ with $\# \Lambda\ge 2$ and 
subgraphs $\Gamma_\alpha$ ($\alpha\in \Lambda$) of $\Gamma$ 
such that $\Gamma=\ast_{\alpha \in \Lambda}\Gamma_\alpha$. 
This is called a \textit{join decomposition} of $\Gamma$ into factors $\Gamma_\alpha$ ($\alpha\in \Lambda$). 
$\Gamma$ is said to be \textit{indecomposable} (as a join) if it is not decomposable. 

A simple decomposable graph $\Gamma$ is called a \textit{cone} 
if $\Gamma$ has a join decomposition into a subgraph 
consisting of only one vertex $v_0$ and a subgraph $\Gamma'$ 
$$\Gamma=\{v_0\}\ast \Gamma'.$$ 
\end{definition}

\begin{remark}
Any simple graph $\Gamma$ is indecomposable as a join if and only if its complement graph $\Gamma^c$ is connected.
\end{remark}

The following is a well-known fact. See Figures \ref{new_ex_fig} and \ref{traditional_fig}. 
\begin{lemma}\label{join-decomposition_lem}
Let $\Gamma$ be a simple graph. Suppose that $\Gamma$ is decomposable. 
Then, $\Gamma$ has a unique join decomposition into indecomposable factors 
$$\Gamma=\ast_{\alpha\in \Lambda}\Gamma_\alpha.$$
\end{lemma}
\begin{proof}
Consider the decomposition of $\Gamma^c$ into connected components 
$$\Gamma^c=\bigsqcup_{\alpha\in \Lambda}(\Gamma^c)_\alpha.$$
Set $V_\alpha:=V((\Gamma^c)_\alpha)$ and 
define $\Gamma_\alpha$ as the subgraph of $\Gamma$ spanned by $V_\alpha$. 
Then, $(\Gamma_\alpha)^c=(\Gamma^c)_\alpha$. 
Additionally, we have a join decomposition $\Gamma=\ast_{\alpha \in \Lambda}\Gamma_\alpha$. 
\end{proof}

\begin{remark}
Any decomposable graph is not a cone 
if and only if each of its indecomposable factors has at least two vertices. 
\end{remark}

\section{Clique-cube complexes and actions on them}\label{clique}
In this section, we consider clique-cube complexes and the actions on them 
following \cite{Charney}. 
Let $A_\Gamma$ be an Artin group associated with a defining graph $\Gamma$ 
as in Subsection \ref{defining_graph}. 
By a theorem of van der Lek (see \cite{Paris}, \cite{van}), for any subset $U\subset V=V(\Gamma)$, 
the subgroup of $A_\Gamma$ generated by $U$ is itself an Artin group associated with 
the full subgraph of $\Gamma$ spanned by $U$. 
We denote this subgroup as $A_U$. 
When $U$ is empty, we define $A_\emptyset =\{1\}$. 
We say that $U$ \textit{spans a clique} in $\Gamma$ if any two elements of $U$ are joined by an edge in $\Gamma$. 

\begin{definition}[\cite{Charney} Definition 2.1]\label{clique cube}
Consider the set 
$$\Delta_\Gamma = \{U \subset V |\ U\text{ spans a clique in }\Gamma,\text{ or }U = \emptyset\}.$$ 
The \textit{clique-cube complex} $C_\Gamma$ is the cube complex 
whose vertices (i.e., $0$-dimensional cubes) are cosets $gA_U$ ($g\in A_\Gamma, U\in \Delta_\Gamma$), 
where two vertices $gA_U$ and $hA_{U'}$ 
are joined by an edge (i.e., a $1$-dimensional cube) in $C_\Gamma$ 
if and only if $gA_U\subset hA_{U'}$ and $U$ and $U'$ differ by a single generator. 
Note that, in this case, 
we can always replace $h$ by $g$, that is, $hA_{U'}=gA_{U'}$. 
More generally, two vertices $gA_U$ and $gA_{U'}$ with $gA_U\subset gA_{U'}$ 
span a $\#(U'\setminus U)$-dimensional cube $[gA_U , gA_{U'}]$ in $C_\Gamma$. 
\end{definition}

The group $A_\Gamma$ acts on the clique-cube complex $C_\Gamma$ by left multiplication, 
$h\cdot gA_U = (hg)A_U$. 
This action preserves the cubical structure and is isometric. 
The action is also co-compact with a fundamental domain $\bigcup_{U\in \Delta_\Gamma}[A_\emptyset, A_U]$, 
where $[A_\emptyset, A_U]$ is a $\# U$-dimensional cube spanned by two vertices $A_\emptyset$ and $A_U$ in $C_\Gamma$. 
However, the action is not proper. 
In fact, the stabilizer of a vertex $gA_U$ is the conjugate subgroup 
$gA_U g^{-1}$, so all vertices except translations of $A_\emptyset$ have infinite stabilizers. 
We also note that $C_\Gamma$ is not a proper metric space because it contains infinite valence vertices. 
Additionally, $C_\Gamma$ has infinite diameter if and only if $\Gamma$ itself is not a clique. 

\begin{remark}[\cite{Charney} Section 2]\label{label}
Each edge in $C_\Gamma$ can be labeled with a generator in $V$. 
For example, the edge between $gA_U$ and $gA_{U\sqcup \{v\}}$ is labeled by $v$. 
Any two parallel edges in a cube have the same label, 
so we can also label the hyperplane dual to such an edge by $v$ 
and say that such a hyperplane is \textit{of $v$-type}. 
Every hyperplane of $v$-type is the translation of a hyperplane 
dual to the edge between $A_\emptyset$ and $A_{\{v\}}$. 
If a hyperplane of $v$-type crosses another hyperplane of $v'$-type, then $(v,v')\in E(\Gamma)$. 
In particular, two different hyperplanes of the same type do not cross each other.
\end{remark}

\begin{theorem}[\cite{Charney} Theorem 2.2]
The clique-cube complex $C_\Gamma$ is CAT(0) for any graph $\Gamma$.
\end{theorem} 

\begin{lemma}[\cite{Charney} Lemma 2.3]\label{linkGamma}
In the clique-cube complex $C_\Gamma$, the link of the vertex 
$A_\emptyset$ is isomorphic to 
the flag simplicial complex whose $1$-skeleton is $\Gamma$.
\end{lemma}

\begin{lemma}[\cite{Charney} Lemma 2.4]
If the clique-cube complex 
$C_\Gamma$ is reducible, that is, 
decomposable as a product of two subcomplexes, then $\Gamma$ is decomposable (as a join). 
In particular, if $\Gamma$ is indecomposable, then $C_\Gamma$ is irreducible.
\end{lemma}

More strongly, we can show the following. 
This proposition is not directly used in the proof of Theorem \ref{main}, 
but is of independent interest. 
\begin{proposition}\label{C_reducible_prop}
The following are equivalent:
\begin{itemize}
\item[(1)] $C_{\Gamma}$ is reducible; 
\item[(2)] $A_\Gamma$ is reducible, that is, $\Gamma^t$ is connected. 
In other words, 
$\Gamma$ can be decomposed as a join of two subgraphs such that all edges between them are labeled by $2$; 
\item[(3)] In addition to (2), $C_{\Gamma}$ is a direct product of $C_{\Gamma'}$ and $C_{\Gamma''}$
when $\Gamma$ is decomposed as a join of two subgraphs $\Gamma'$ and $\Gamma''$ 
such that all edges between them are labeled by $2$. 
\end{itemize}
\end{proposition}

\begin{proof}
$(3)\Rightarrow (1)$ is obvious. We show that $(1)\Rightarrow (2)\Rightarrow (3)$.

We first consider $(1)\Rightarrow (2)$. Suppose that $C_{\Gamma}$ is reducible.
First, we show that $\Gamma$ can be decomposed as a join of two subgraphs.
We fix two subcomplexes $C'$ and $C''$ of $C_{\Gamma}$ satisfying $C_{\Gamma}=C'\times C''$.
Then, for any vertex $v=(v', v'')$ of $C_{\Gamma}=C'\times C''$, 
$\mathrm{Lk}_{C_{\Gamma}}v$ is the join of $\mathrm{Lk}_{C'}v'$ and $\mathrm{Lk}_{C''}v''$.
Let $\Gamma_v$, $\Gamma'_v$, and $\Gamma''_v$ be $1$-skeletons of 
$\mathrm{Lk}_{C}v$, $\mathrm{Lk}_{C'}v'$, and $\mathrm{Lk}_{C''}v''$, respectively.
Then, $\Gamma_v$ is the join of $\Gamma'_v$ and $\Gamma''_v$. 
In particular, $\Gamma_{A_{\emptyset}}$ is 
the join of $\Gamma'_{A_{\emptyset}}$ and $\Gamma''_{A_{\emptyset}}$. 
We set $\Gamma':=\Gamma'_{A_{\emptyset}}$ and $\Gamma'':=\Gamma''_{A_{\emptyset}}$. 
Because $\Gamma$ is isomorphic to $\Gamma_{A_{\emptyset}}$ by Lemma \ref{linkGamma}, 
$\Gamma$ can be regarded as the join of $\Gamma'$ and $\Gamma''$. 
Thus far, the argument is based on th proof of Lemma \ref{linkGamma} in \cite{Charney}. 

Next, we show that all edges between $\Gamma'$ and $\Gamma''$ are labeled by $2$.
Assume that we have $s\in V(\Gamma')$ and $t\in V(\Gamma'')$ 
such that $e=(s,t)$ is an edge of $\Gamma$ with label $m>2$. 
Let us consider three squares 
$[A_{\emptyset}, A_{\{s,t\}}]$, $[sA_{\emptyset}, A_{\{s,t\}}]$, and $[tA_{\emptyset}, A_{\{s,t\}}]$ 
around $A_{\{s,t\}}$. Then, we have a $3$-line subgraph of $\Gamma_{A_{\{s,t\}}}$ 
corresponding to these three squares (see Figure \ref{3-line_fig}). 
\begin{figure}
\begin{center}
\includegraphics[width=8cm,pagebox=cropbox,clip]{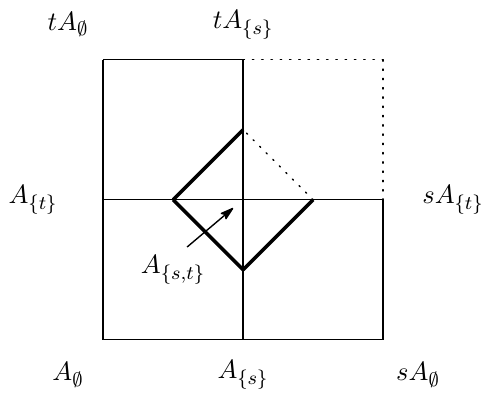}
\caption{A $3$-line full subgraph of $\Gamma_{A_{\{s,t\}}}$.}\label{3-line_fig}
\end{center}
\end{figure}
The $3$-line subgraph is a full subgraph of $\Gamma_{A_{\{s,t\}}}$, 
because it follows from Lemma \ref{toy} (see the next section) that 
there is no square in $C_{\Gamma}$ 
containing both edges $[A_{\{s,t\}},tA_{\{s\}}]$ and $[A_{\{s,t\}},sA_{\{t\}}]$. 
Because $C_\Gamma$ is $C'\times C''$ and edges $[A_{\emptyset}, A_{\{s\}}]$ and $[A_{\emptyset}, A_{\{t\}}]$ 
of $C_{\Gamma}$ correspond to edges of $C'\times \{A_\emptyset\}$ and $\{A_\emptyset\}\times C''$, respectively, 
edges $[A_{\{t\}}, A_{\{s,t\}}]$ and $[A_{\{s\}}, A_{\{s,t\}}]$ of $C_{\Gamma}$ must correspond to 
edges of $C'\times \{A_{\{t\}}\}$ and $\{A_{\{s\}}\}\times C''$, respectively. 
Thus, the two middle vertices of the $3$-line full subgraph of $\Gamma_{A_{\{s,t\}}}=\Gamma'_{A_{\{s,t\}}}\ast \Gamma''_{A_{\{s,t\}}}$ 
belong to $\Gamma'_{A_{\{s,t\}}}$ and $\Gamma''_{A_{\{s,t\}}}$, respectively. 
This contradicts the fact that any $3$-line full subgraph of a join of two graphs is contained in either of the join factors.

We now show that $(2)\Rightarrow (3)$. 
Suppose that $\Gamma$ is decomposed as a join of two subgraphs $\Gamma'$ and $\Gamma''$
such that all edges between them are labeled by $2$. Then, we have a bijection 
\begin{equation}\label{delta}
\Delta_{\Gamma'}\times \Delta_{\Gamma''}\to \Delta_\Gamma;(T',T'')\mapsto T'\sqcup T''.
\end{equation}
In addition, because $A_{\Gamma}$ is a direct product of subgroups $A_{\Gamma'}$ and $A_{\Gamma''}$, 
we have a group isomorphism 
\begin{equation}\label{group}
A_{\Gamma'}\times A_{\Gamma''}\to A_\Gamma; (g',g'')\mapsto g'g''. 
\end{equation}
Clearly, (\ref{delta}) and (\ref{group}) imply a bijection from 
vertices of $C_{\Gamma'}\times C_{\Gamma''}$ to vertices of $C_\Gamma$ 
\begin{equation*}
\phi^0:C_{\Gamma'}^0\times C_{\Gamma''}^0\to C_{\Gamma}^0;(g'A_{T'},g''A_{T''})\mapsto g'g''A_{T'\sqcup T''}.
\end{equation*}
This can be extended to the cubical isomorphism 
\begin{equation*}
\phi:C_{\Gamma'}\times C_{\Gamma''}\to C_{\Gamma} 
\end{equation*}
such that $\phi([g'A_{T'}, g'A_{U'}]\times [g''A_{T''}, g''A_{U''}])=[g'g''A_{T'\cup T''}, g'g''A_{U'\cup U''}]$ 
for any $(T',T'')\in \Delta_{\Gamma'}\times \Delta_{\Gamma''}$ and 
$(g',g'')\in A_{\Gamma'}\times A_{\Gamma''}$. 
\end{proof}

\begin{lemma}[\cite{Charney} Lemma 3.2]\label{CM Lemma 3.2}
Suppose that $\Gamma$ is not a cone. 
Then, the action of $A_\Gamma$ on $C_\Gamma$ is minimal. 
That is, for any point $x\in C_\Gamma$, we have 
$Hull(A_\Gamma x)=C_\Gamma$ 
(the convex hull of the orbit of $x$ is all of $C_\Gamma$).
\end{lemma}

\begin{proposition}\label{finite normal subgroup}
Suppose that $\Gamma$ is not a cone. 
Then, a finite normal subgroup of $A_\Gamma$ is trivial. 
In particular, a finite center of  $A_\Gamma$ is trivial. 
Also, if $A_\Gamma$ is isomorphic to a direct product $A_1\times A_2$ and 
$A_1$ is finite, then $A_1$ is trivial. 
\end{proposition}
\begin{proof}
Let $N$ be a finite normal subgroup of $A_\Gamma$. 
Set $$Fix(N):=\{x\in C_\Gamma | \ nx=x\text{ for any }n\in N \}.$$ 
Because $N$ is finite and $C_\Gamma$ is a complete CAT(0) space, we have $Fix(N)\neq \emptyset$.
Take any $x\in Fix(N)$. 
Then, $A_\Gamma x\subset Fix(N)$. 
Indeed, the normality of $N$ implies that, for any $g\in A_\Gamma$ and $n\in N$, there exists $n'\in N$ 
such that $ng=gn'$. Therefore, we have  $ngx=gn'x=gx$.
Because $C_\Gamma$ is CAT(0), $Fix(N)$ is convex. Hence, we have $ Hull(A_\Gamma x)\subset Fix(N)$. 
By Lemma \ref{CM Lemma 3.2}, we have $Hull(A_\Gamma x)=C_\Gamma$. 
Hence, $Fix(N)=C_\Gamma$. In particular, $A_\emptyset \in Fix(N)$. 
In general, $A_\emptyset$ is not fixed by any nontrivial element of $A_\Gamma$. 
Hence, $N$ must be trivial.
\end{proof}

\section{Lemmas on local geometry of clique-cube complexes}\label{local_geom}
In this section, we state two lemmas related to the local geometry of clique-cube complexes. 
The first one is used in the proof of Proposition \ref{C_reducible_prop}. 
The second is used in the proof of Theorem \ref{main}. 

Recall that the dihedral group for any $r\in \mathbb{N}$ is defined as 
$$I_2(r):=\left\{
\begin{array}{ll}
\langle s,t \mid s^2=1, t^2=1, st\cdots s =ts\cdots t \ (\text{length } r)\rangle & \text{ if }r\text{ is odd},\\
\langle s,t \mid s^2=1, t^2=1, st\cdots t =ts\cdots s \ (\text{length } r)\rangle & \text{ if }r\text{ is even}. \\
\end{array}
\right.
$$
It is well-known that $\#I_2(r)=2r$. 

Let $A_\Gamma$ be an Artin group associated with a defining graph $\Gamma$ 
as in Subsection \ref{defining_graph}. 

\begin{lemma}\label{toy}
Let $e=(s,t)$ be an edge of $\Gamma$ with label $m$ greater than $2$. 
Then, there exists no square in $C_{\Gamma}$ 
containing both edges $[A_{\{s,t\}},tA_{\{s\}}]$ and $[A_{\{s,t\}},sA_{\{t\}}]$, that is, 
there exist no $p,q\in \mathbb{Z}$ such that $ts^p=st^q$ in $A_\Gamma$. 
\end{lemma}
\noindent See Figure \ref{3-line_fig}. 
\begin{proof}
Assume that we have a square in $C_{\Gamma}$ 
containing both edges $[A_{\{s,t\}},tA_{\{s\}}]$ and $[A_{\{s,t\}},sA_{\{t\}}]$. 
Then, the square has four vertices $A_{\{s,t\}}$, $tA_{\{s\}}$, $sA_{\{t\}}$, and 
$ts^pA_\emptyset=st^qA_\emptyset$ 
where $p,q$ are some integers. 
Then, we have $p=q$. 
Indeed, if we consider the projection $A_{\{s,t\}}\to A_{\{s\}}\cong\mathbb{Z}$ 
defined by $s\mapsto s$ and $t\mapsto s$, 
then $ts^p=st^q$ in $A_{\{s,t\}}$ implies $s^{1+p}=s^{1+q}$ in $A_{\{s\}}$, and thus $p=q$. 
Hence, we have $ts^p=st^p$ in $A_{\{s,t\}}$. 
By the natural projection $A_{\{s,t\}}\to W_{\{s,t\}}$, 
$ts^p=st^p$ in $A_{\{s,t\}}$ implies 
$ts=st$ in $W_{\{s,t\}}$ if $p$ is odd 
and $t=s$ in $W_{\{s,t\}}$ if $p$ is even. 
Both cases contradict the fact that 
$ts\neq st$ and $s\neq t$ in $W_{\{s,t\}}$ 
in the case where $m>2$. 
\end{proof}
\begin{lemma}\label{twist}
Let $e=(s,t)$ be an edge of $\Gamma$ with label $m$ greater than $2$.
Let $\tau$ be an alternating word 
$$\tau:=st\cdots s\left\{
\begin{array}{l}
\text{of length }m \text{ if }m\text{ is odd}, \\
\text{of length }m+1 \text{ if }m\text{ is even}.\\
\end{array}
\right.
$$
Let $U\in \Delta_\Gamma$ with $s,t\in U$. 
Then, there exists no square in $C_{\Gamma}$ 
containing both edges $[A_U,A_{U\setminus\{s\}}]$ and $[A_U,\tau A_{U\setminus\{t\}}]$, that is, 
there exists no $g\in A_{U\setminus\{s\}}$ such that $gA_{U\setminus\{t\}}=\tau A_{U\setminus\{t\}}$.
Additionally, there exists no square in $C_{\Gamma}$ 
containing both edges $[A_U,A_{U\setminus\{t\}}]$ and $[A_U,\tau A_{U\setminus\{s\}}]$, that is, 
there exists no $g\in A_{U\setminus\{t\}}$ such that $gA_{U\setminus\{s\}}=\tau A_{U\setminus\{s\}}$.
\end{lemma}
\noindent See Figure \ref{square_fig}. 
\begin{figure}
\begin{center}
\includegraphics[width=8cm,pagebox=cropbox,clip]{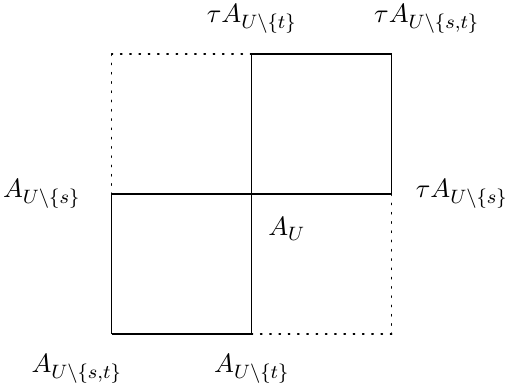} 
\caption{Squares around $A_{U}$.}\label{square_fig}
\end{center}
\end{figure}
\begin{proof}
We prove this lemma by induction on $k=\#U$.

(1) The base case where $k=2$. 
In this case, $U=\{s,t\}$.
It is sufficient to show that, for any $p\in \mathbb{Z}$, 
$t^p A_{U\setminus\{t\}}\neq \tau A_{U\setminus\{t\}}$ and $s^p A_{U\setminus\{s\}}\neq \tau A_{U\setminus\{s\}}$.

(1-1) Suppose that $m$ is odd. 
Note that $\tau=st\cdots s$ is equal to an alternating word $ts\cdots t$ of length $m$ in 
$$A_{\{s,t\}}=\langle s,t \mid st\cdots s =ts\cdots t \ (\text{length } m)\rangle= A_U\subset  A_\Gamma.$$ 
We assume that there exists $p\in \mathbb{Z}$ 
such that $t^p A_{U\setminus\{t\}}=\tau A_{U\setminus\{t\}}$. 
Then, we have $t^{-p}\tau \in A_{U\setminus\{t\}}$. 
In contrast, we clearly have that $t^{-p}\tau\in A_{\{s,t\}}$. 
Because $U\setminus\{t\}=\{s\}$, 
there exists $q\in \mathbb{Z}$ such that $t^ps^q=\tau$ in $A_{\{s,t\}}$.
Then, we have $p+q=m$. 
Indeed, if we consider the projection $A_{\{s,t\}}\to A_{\{s\}}\cong\mathbb{Z}$ 
defined by $s\mapsto s$ and $t\mapsto s$, 
then $t^ps^q=\tau$ in $A_{\{s,t\}}$ implies $s^{p+q}=s^m$ in $A_{\{s\}}$, and thus $p+q=m$. 
Hence, we have $t^ps^{m-p}=st\cdots s=ts\cdots t$ in $A_{\{s,t\}}$. 

(1-1-1) 
We treat the case where $p$ is odd. 
By the natural projection $A_{\{s,t\}}\to W_{\{s,t\}}$,
$t^ps^{m-p}=ts\cdots t$ in $A_{\{s,t\}}$ and 
$t^ps^{m-p}=t$ in $W_{\{s,t\}}$ imply 
$t=ts\cdots t$ in $W_{\{s,t\}}$. 
Thus, we have $1=s\cdots t$ in $W_{\{s,t\}}$. 
This means that we have a projection $I_2((m-1)/2)\to W_{\{s,t\}}\cong I_2(m)$ 
defined by $s\mapsto s$ and $t\mapsto t$. 
Thus, we have $m-1=\#I_2((m-1)/2)\ge \#I_2(m)=2m$. This contradicts $m> 2$. 

(1-1-2) We treat the case where $p$ is even. 
In this case, $t^ps^{m-p}=st\cdots s$ in $A_{\{s,t\}}$ and 
$t^ps^{m-p}=s$ in $W_{\{s,t\}}$ imply 
$s=st\cdots s$ in $W_{\{s,t\}}$. 
Thus, we have $1=s\cdots t$ in $W_{\{s,t\}}$. 
This means that we have a projection $I_2((m-1)/2)\to W_{\{s,t\}}\cong I_2(m)$ 
defined by $s\mapsto s$ and $t\mapsto t$. 
Thus, we have $m-1=\#I_2((m-1)/2)\ge \#I_2(m)=2m$. This contradicts $m> 2$. 

By (1-1-1) and (1-1-2), for any $p\in \mathbb{Z}$, we have $t^p A_{U\setminus\{t\}}\neq \tau A_{U\setminus\{t\}}$. 
By the same argument, for any $p\in \mathbb{Z}$, we have 
$s^p A_{U\setminus\{s\}}\neq \tau A_{U\setminus\{s\}}$.

(1-2) Suppose that $m$ is even. 
Note that alternating words $st\cdots t$ and $ts\cdots s$ of length $m$ are equal in 
$$A_{\{s,t\}}=\langle s,t \mid st\cdots t =ts\cdots s \ (\text{length } m)\rangle= A_U\subset  A_\Gamma.$$ 
We assume that there exists $p\in \mathbb{Z}$ 
such that $t^p A_{U\setminus\{t\}}=\tau A_{U\setminus\{t\}}$. 
Then, from the same argument as in (1-1), 
there exists $q\in \mathbb{Z}$ such that $t^ps^q=\tau$ in $A_{\{s,t\}}$.
Thus, we have $p+q=m+1$. 
Indeed, if we consider the projection $A_{\{s,t\}}\to A_{\{s\}}\cong\mathbb{Z}$ 
defined by $s\mapsto s$ and $t\mapsto s$, 
then $t^ps^q=\tau$ in $A_{\{s,t\}}$ implies $s^{p+q}=s^{m+1}$ in $A_{\{s\}}$, and thus $p+q=m+1$. 
Hence, we have $t^ps^{m+1-p}=st\cdots s$ in $A_{\{s,t\}}$. 
This implies $t^ps^{m-p}=st\cdots t$ in $A_{\{s,t\}}$. 
Thus, we have $t^ps^{m-p}=st\cdots t=ts\cdots s$ in $A_{\{s,t\}}$. 

(1-2-1) We treat the case where $p$ is odd. 
By the natural projection $A_{\{s,t\}}\to W_{\{s,t\}}$,
$t^ps^{m-p}=ts\cdots s$ in $A_{\{s,t\}}$ and 
$t^ps^{m-p}=ts$ in $W_{\{s,t\}}$ imply 
$ts=ts\cdots s$ in $W_{\{s,t\}}$. 
Thus, $1$ and an alternating word $s\cdots t$ of length $m-2$ are equal in $W_{\{s,t\}}$. 
This means that we have a projection $I_2((m-2)/2)\to W_{\{s,t\}}\cong I_2(m)$ 
defined by $s\mapsto s$ and $t\mapsto t$. 
Thus, we have $m-2=\#I_2((m-2)/2)\ge \#I_2(m)=2m$. This contradicts $m> 2$. 

(1-2-2)
We treat the case where $p$ is even. 
In this case, $t^ps^{m-p}=st\cdots t$ in $A_{\{s,t\}}$ and 
$t^ps^{m-p}=1$ in $W_{\{s,t\}}$ implies 
$1=st\cdots t$ in $W_{\{s,t\}}$. 
This means that we have a projection $I_2(m/2)\to W_{\{s,t\}}\cong I_2(m)$ 
defined by $s\mapsto s$ and $t\mapsto t$. 
Thus, we have $m=\#I_2(m/2)\ge \#I_2(m)=2m$. This contradicts $m> 2$. 

By (1-2-1) and (1-2-2), for any $p\in \mathbb{Z}$, we have $t^p A_{U\setminus\{t\}}\neq \tau A_{U\setminus\{t\}}$. 
By the same argument, for any $p\in \mathbb{Z}$, we have 
$s^p A_{U\setminus\{s\}}\neq \tau A_{U\setminus\{s\}}$.

(2) Suppose that $k>2$ and the statement is true for $k-1$. 
Let $U=\{u_1, \ldots, u_{k-2}, s, t\}$, where $\# U=k$. 
Assume that there exists $g\in A_{U\setminus\{s\}}$ such that $gA_{U\setminus\{t\}}=\tau A_{U\setminus\{t\}}$.
Then, we have a square $[gA_{U\setminus\{s,t\}}, A_U]$ in $C_{\Gamma}$. 
Its vertices are 
$A_U$, $\tau A_{U\setminus\{t\}}$, $gA_{U\setminus\{s,t\}}$, and $A_{U\setminus\{s\}}$. 
Because $C_{\Gamma}$ is CAT(0), 
this square spans a cube together with other two squares $[A_{U\setminus\{u_1, s\}}, A_{U}]$ and $[\tau A_{U\setminus\{u_1, t\}}, A_{U}]$. 
See Figure \ref{local_fig}. 

\begin{figure}
\begin{center}
\includegraphics[width=12cm,pagebox=cropbox,clip]{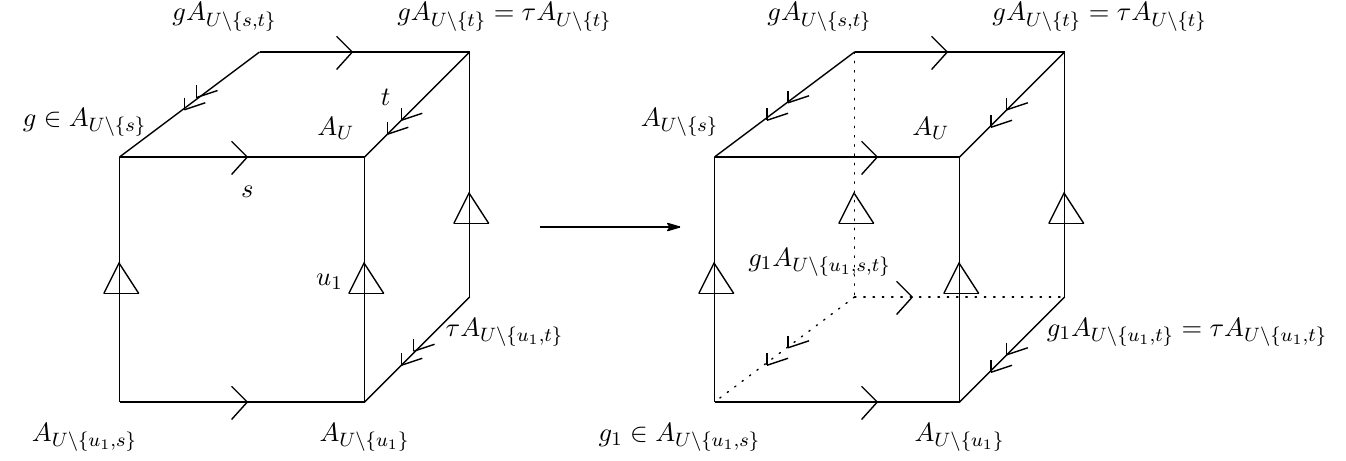} 
\caption{Three squares around $A_U$ span a cube.}\label{local_fig}
\end{center}
\end{figure}

This cube contains another square $[g_1A_{U\setminus\{u_1, s, t\}}, A_{U\setminus\{u_1\}}]$ as a face, 
where $g_1\in A_{U\setminus\{u_1, s\}}=A_{(U\setminus\{u_1\})\setminus\{s\}}$. 
Then, we have $g_1 A_{(U\setminus\{u_1\})\setminus\{t\}}=\tau A_{(U\setminus\{u_1\})\setminus\{t\}}$. 
Because $\#(U\setminus\{u_1\})=k-1$, 
this contradicts the inductive assumption.
\end{proof}

\section{Proof of Theorem \ref{main}}\label{main_proof}
In this section, let $A_\Gamma$ be an Artin group associated with a graph 
$\Gamma$ that has at least three vertices, and suppose that $\Gamma$ is not a cone. 
We show our main theorem (Theorem~\ref{main}), that is, the following are equivalent:

\begin{enumerate}
\item $A_\Gamma$ is irreducible, that is, 
$\Gamma$ cannot be decomposed as a join of two subgraphs such that all edges between them are labeled by $2$; 
\item $A_\Gamma$ has a WPD contracting element with respect to the isometric action on the clique-cube complex;
\item $A_\Gamma$ is acylindrically hyperbolic;
\item $A_\Gamma$ is directly indecomposable, that is, it cannot be decomposed as a direct product 
of two nontrivial subgroups. 
\end{enumerate}

\subsection{Proof of $(2)\Rightarrow(3),(3)\Rightarrow(4),(4)\Rightarrow(1)$ in Theorem \ref{main}}\label{part}
We show that $(2)\Rightarrow(3),(3)\Rightarrow(4),(4)\Rightarrow(1)$ in Theorem \ref{main}. 

First, $(2)\Rightarrow(3)$ follows from Theorem \ref{BBF}. 

Next, we show $(3)\Rightarrow(4)$. 
Let $A_\Gamma$ be acylindrically hyperbolic. If $A_\Gamma$ is isomorphic to a direct product $A_1\times A_2$, then 
either $A_1$ or $A_2$ is finite by acylindrical hyperbolicity \cite[Corollary 7.3]{Osin}. 
Proposition \ref{finite normal subgroup} implies that the finite one is trivial. 
Hence, $A_\Gamma$ is directly indecomposable. 

Finally, $(4)\Rightarrow(1)$ is clear. 

\subsection{Proof of $(1)\Rightarrow(2)$ in Theorem \ref{main}}
In this subsection, we give a proof of $(1)\Rightarrow(2)$ in Theorem \ref{main}. 
Suppose that $A_\Gamma$ is irreducible. 
We will construct an element $\gamma\in A_\Gamma$ and 
show that $\gamma$ is a WPD contracting element with respect to the action on the clique-cube complex. 
When $\Gamma$ is indecomposable, that is, not a join, a WPD contracting element is already given in \cite[Remark 4.5]{Charney}. 
Hence, it is sufficient to treat the case where $\Gamma$ is decomposable. 
From now on, we suppose that $\Gamma$ is decomposable. 

We consider a unique join decomposition $\Gamma=\ast_{1\le i\le k}\Gamma_i$ 
into indecomposable factors (see Lemma \ref{join-decomposition_lem}). 
We set $V_i=V(\Gamma_i)$ and $E_i=E(\Gamma_i)$ for each $i\in\{1,\ldots,k\}$. 
Then, for every $i\in\{1,\ldots,k\}$, the complement graph $(\Gamma_i)^c$ of $\Gamma_i$ 
is connected and the complement graph $\Gamma^c$ of $\Gamma$ is 
the disjoint union of connected components $\Gamma^c=\bigsqcup_{1\le i \le k}(\Gamma_i)^c$. 

We define $Q(\Gamma)$ as a finite simple graph with 
\begin{itemize}
\item the vertex set $V(Q(\Gamma))=\{V_i\}_{1\leq i \leq k}$, and
\item the edge set $E(Q(\Gamma))=\{(V_i,V_j), (V_j,V_i) \mid 1\le i<j\le k,(\Gamma_i\ast\Gamma_j)^t\text{ is connected}\}$. 
\end{itemize}
Note that $(\Gamma_i)^t$ is connected because $(\Gamma_i)^c$ is connected, 
$V((\Gamma_i)^c)=V((\Gamma_i)^t)=V_i$, and $E((\Gamma_i)^c)\subset E((\Gamma_i)^t)$. 
Note that the following are equivalent for different $i, j\in \{1,\ldots,k\}$: 
\begin{itemize}
\item $(\Gamma_i\ast\Gamma_j)^t$ is connected;
\item there exists an edge with label greater than $2$ between a vertex of $\Gamma_i$ and a vertex of $\Gamma_j$.
\end{itemize}

We confirm our setting: 
\begin{enumerate}
\item $\Gamma$ is decomposable, i.e., $\Gamma^c$ is not connected, that is, $k\ge 2$; 
\item $\Gamma$ is not a cone, i.e., every connected component of $\Gamma^c$ has at least two vertices, 
that is, $\#\Gamma_i\ge 2$ for any $i\in\{1,\ldots,k\}$; 
\item $A_\Gamma$ is irreducible, i.e., $\Gamma^t$ is connected, that is, $Q(\Gamma)$ is connected. 
\end{enumerate}

We take a spanning tree $T$ of $Q(\Gamma)$. 
We regard $T$ as a rooted tree with the root $V_1$.
By trading indices of $V_2,\ldots,V_k$ if necessary, 
we suppose that $i<j$ only if $V_i$ is not farther than $V_j$ from $V_1$ in $T$. 
For each $i,j$ with $i<j$ and $(V_i,V_j)\in E(T)$, 
take an edge $e_{i,j}\in E(\Gamma)$ with label $m_{i,j}=\mu(e_{i,j})$ greater than $2$ 
with $s_{i,j}:=s_{e_{i,j}}\in V(\Gamma_i)$ and $t_{i,j}:=t_{e_{i,j}}\in V(\Gamma_j)$, 
and set $e_{j,i}:=\iota(e_{i,j})$. 
For any $i,j\in \{1,\ldots,k\}$ with $(V_i,V_j)\in E(T)$, 
let $\tau_{i,j}$ be an alternating word 
$$\tau_{i,j}=s_{i,j}t_{i,j}\cdots s_{i,j}$$ of length $m_{i,j}$ if $m_{i,j}$ is odd, 
and let $\tau_{i,j}$ be an alternating word 
$$\tau_{i,j}=s_{i,j}t_{i,j}\cdots s_{i,j}$$ of length $m_{i,j}+1$ if $m_{i,j}$ is even. 

\begin{lemma}\label{closed}
There exist $n\in \mathbb{N}$, 
for each $i\in \{1,\ldots,k\}$, 
a closed path 
$(v_{i,1},\ldots,v_{i, n},v_{i, n+1})$ with $v_{i,1}=v_{i,n+1}$ on $(\Gamma_{i})^c$ 
passing through every vertex at least once, 
and for any $i, j\in \{1,\ldots,k\}$ with $(V_i,V_j)\in E(T)$, 
$l(i,j)\in\{1,\dots,n\}$ such that $(v_{i,l(i,j)},v_{j,l(i,j)})=(s_{i,j},t_{i,j})(=e_{i,j})$ 
and $l(j,i)=l(i,j)$. 
\end{lemma}
\begin{proof}
For any $i\in \{1,\ldots,k\}$, 
we consider the minimum length $n_i$ of closed paths 
on $(\Gamma_{i})^c$ passing through every vertex at least once. 
Set $n=\prod_{1\le i\le k}n_i$. 
Then, for any $i\in \{1,\ldots,k\}$, 
by concatenating $\frac{n}{n_i}$ copies 
of a closed path of length $n_i$ on $(\Gamma_{i})^c$ 
passing through every vertex at least once, 
we have a closed path $(v'_{i,1},\ldots,v'_{i,n},v'_{i, n+1})$ 
with $v'_{i,1}=v'_{i,n+1}$ on $(\Gamma_{i})^c$ 
passing through every vertex at least $\frac{n}{n_i}$ times (in particular, at least once). 

We set $(v_{1,1},\ldots,v_{1, n},v_{1, n+1}):=(v'_{1,1},\ldots,v'_{1, n},v'_{1, n+1})$.
For any $j\in \{2,\ldots,k\}$, we define 
$(v_{j,1},\ldots,v_{j, n},v_{j, n+1})$ inductively as follows. 
Take $j\in \{2,\ldots,k\}$. 
Assume that $(v_{i,1},\ldots,v_{i, n},v_{i, n+1})$ 
is defined for $i\in \{1,\ldots, j-1\}$ with $(V_i,V_j)\in E(T)$. 
Then, we set $l(i,j)$ as the minimum $l$ such that $v_{i,l}=s_{i,j}$. 
By a cyclic permutation of  $v'_{j,1},\ldots,v'_{j, n}$, 
we have $v_{j,1},\ldots,v_{j, n}$ such that $v_{j,l(i,j)}=t_{i,j}$. 
By setting $v_{j,n+1}:=v_{j,1}$, we have 
$(v_{j,1},\ldots,v_{j, n},v_{j, n+1})$. 
For any $i,j$ with $1\le i<j\le k$ and $(V_i,V_j)\in E(T)$, we set $l(j,i):=l(i,j)$. 
\end{proof}

We take $n\in \mathbb{N}$, $(v_{i,1},\ldots,v_{i, n},v_{i, n+1})$ for each $i\in \{1,\ldots,k\}$,
and $l(i,j)\in\{1,\dots,n\}$ for any $i, j\in \{1,\ldots,k\}$ with $(V_i,V_j)\in E(T)$
as in Lemma~\ref{closed}.
For $1\leq l\le n$, set 
$$\lambda_l:=v_{1,l}v_{2,l}\cdots v_{k,l}.$$
For any $i, j\in \{1,\ldots,k\}$ with $(V_i,V_j)\in E(T)$, we define 
$\lambda_l(i,j)$ as 
$\lambda_l(i,j):=\lambda_l$ if $l\neq l(i,j)$. In addition, we define $\lambda_{l(i,j)}(i,j)$ as 
$$\lambda_{l(i,j)}(i,j):=
\begin{array}{lll}
\tau_{i,j} v_{1,l(i,j)}v_{2,l(i,j)}\cdots v_{i-1,l(i,j)}v_{i+1,l(i,j)} & &\\
\ \ \cdots v_{j-1,l(i,j)}v_{j+1,l(i,j)}\cdots v_{k,l(i,j)}
\end{array}
$$
if $i<j$, and $\lambda_{l(i,j)}(i,j):=\lambda_{l(j,i)}(j,i)$ if $i>j$, 
where we note that $l(j,i)=l(i,j)$. 
Moreover, we define 
$\gamma(i,j)$ as 
\begin{align*}
\gamma(i,j):&=\lambda_1(i,j)\lambda_2(i,j)\cdots \lambda_{n}(i,j) \\
&=\lambda_{1}\lambda_{2}\cdots \lambda_{l(i,j)-1}\lambda_{l(i,j)}(i,j) \lambda_{l(i,j)+1}\cdots \lambda_{n}.
\end{align*}
Then, we have $\gamma(j,i)=\gamma(i,j)$. 

Take a closed path 
\begin{equation}\label{pathtree}
(V_{i_1},\dots,V_{i_r},V_{i_{r+1}})
\end{equation}
with $i_1=i_{r+1}=1$ 
on the spanning tree $T$ of $Q(\Gamma)$ 
passing through every vertex at least once. 
We define $\gamma$ as 
\begin{equation}\label{gammaelement}
\gamma:=\gamma(i_1,i_2)\gamma(i_2,i_3)\cdots\gamma(i_r,i_{r+1}).
\end{equation}
We set 
$$\gamma(0):=1, \gamma(1):=\lambda_1(i_1,i_2),\gamma(2):=\lambda_1(i_1,i_2)\lambda_2(i_1,i_2),\ldots,\gamma(n):=\gamma(i_1,i_2)$$ 
and for $a\in\{2,\ldots r\}$ and $l\in\{1,\ldots,n\}$, we set 
$$\gamma((a-1)n+l):=\gamma(i_1,i_2)\cdots\gamma(i_{a-1},i_a)(\lambda_1(i_a,i_{a+1})\cdots\lambda_l(i_a,i_{a+1})).$$
In particular, we have $\gamma(rn)=\gamma$. 

We confirm the following for convenience: 
\begin{enumerate}
\item $k$ is the number of indecomposable factors of a unique join decomposition of $\Gamma$; 
\item $n$ is the common length of the closed paths on $(\Gamma_i)^c$ for all $i\in\{1,\ldots,k\}$ taken in Lemma \ref{closed}; 
\item $r$ is the length of the closed path (\ref{pathtree}) on $T$ taken above. 
\end{enumerate}

For $1\leq l\le n$, set 
$$U_l:=\{v_{1,l},v_{2,l},\ldots, v_{k,l}\}.$$
Then, $U_l$ spans a clique in $\Gamma$, that is, $U_l\in \Delta_\Gamma$. 
Hence, we have a $k$-dimensional cube $[A_\emptyset, A_{U_l}]$ in $C_\Gamma$. 
The hyperplanes dual to edges of $[A_\emptyset,A_{U_l}]$ are of $v_{i,l}$-type ($i\in \{1,\ldots,k\}$).  
We denote such hyperplanes as $H_{i,l}$ ($i\in \{1,\ldots,k\}$). 

\begin{lemma}\label{key0}
For $i\in \{1,\ldots,k\}$, $l\in \{1,\ldots,n\}$, and $a\in\{1,\ldots,r\}$, we have the following: 
\begin{enumerate}
\item $H_{i,l}\cap H_{i,l+1}=\emptyset$, where we set $H_{i,n+1}:=H_{i,1}$; 
\item $H_{i,l}\cap \lambda_l H_{i,l}=\emptyset$; 
\item $H_{i,l(i_{a},i_{a+1})}\cap \lambda_{l(i_{a},i_{a+1})}(i_{a},i_{a+1})H_{i,l(i_{a},i_{a+1})}=\emptyset$; 
\item $[A_\emptyset,A_{U_l}]\cap [A_\emptyset,A_{U_{l+1}}]=\{A_\emptyset\}$, where we set $U_{n+1}:=U_{1}$;
\item $[A_\emptyset,A_{U_l}]\cap \lambda_l [A_\emptyset,A_{U_l}]=\{A_{U_l}\}$;
\item $[A_\emptyset,A_{U_{l(i_{a},i_{a+1})}}]\cap \lambda_{l(i_{a},i_{a+1})}(i_{a},i_{a+1})[A_\emptyset,A_{U_{l(i_{a},i_{a+1})}}]=\{A_{U_{l(i_{a},i_{a+1})}}\}$.
\end{enumerate}
\end{lemma}
\begin{proof}
(1) $H_{i,l}$ and $H_{i,l+1}$ are of $v_{i,l}$-type and $v_{i,l+1}$-type, respectively. 
Note that $v_{i,l}\neq v_{i,l+1}$ implies $H_{i,l}\neq H_{i,l+1}$. 
Because  $(v_{i,l}, v_{i,l+1})\notin E(\Gamma)$, we have $\{v_{i,l}, v_{i,l+1}\}\notin \Delta_\Gamma$, 
and thus $H_{i,l}\cap H_{i,l+1}=\emptyset$. 

(2) Because $[A_{U_l\setminus\{v_{i,l}\}},A_{U_l}]\subset N(H_{i,l})$, 
we have $[\lambda_l A_{U_l\setminus\{v_{i,l}\}},\lambda_l A_{U_l}]\subset \lambda_l N(H_{i,l})$. 
Note that $\lambda_l A_{U_l}=A_{U_l}$ and $\lambda_l A_{U_l\setminus\{v_{i,l}\}}=v_{1,l}\cdots v_{i,l} A_{U_l\setminus\{v_{i,l}\}}$. 
Now, assume that $H_{i,l}\cap \lambda_l H_{i,l}\neq \emptyset$. 
Because $H_{i,l}$ and $\lambda_l H_{i,l}$ are of $v_{i,l}$-type, we see that $H_{i,l}= \lambda_l H_{i,l}$ (see Remark \ref{label}). 
Then, we have 
$$[\lambda_l A_{U_l\setminus\{v_{i,l}\}},\lambda_l A_{U_l}]=[v_{1,l}\cdots v_{i,l} A_{U_l\setminus\{v_{i,l}\}},A_{U_l}] \subset N(H_{i,l}).$$
Hence, we obtain $A_{U_l\setminus\{v_{i,l}\}}=v_{1,l}\cdots v_{i,l} A_{U_l\setminus\{v_{i,l}\}}$. 
This means that $v_{i,l}\in A_{U_l\setminus\{v_{i,l}\}}$. 
However, from \cite{van}, $\{v_{i,l}\}\cap (U_l\setminus\{v_{i,l}\})=\emptyset$ implies that 
$A_{\{v_{i,l}\}}\cap A_{U_l\setminus\{v_{i,l}\}}=A_\emptyset=\{1\}$. 
This contradicts $v_{i,l}\neq 1$ in $A_\Gamma$.

(3) When $i\neq i_a, i_{a+1}$, by the same argument as in (2), we see that 
$$H_{i,l(i_{a},i_{a+1})}\cap \lambda_{l(i_{a},i_{a+1})}(i_{a},i_{a+1})H_{i,l(i_{a},i_{a+1})}=\emptyset.$$
Now, assume that $i=i_a$ or $i=i_{a+1}$ and
$$H_{i,l(i_a,i_{a+1})}\cap \lambda_{l(i_{a},i_{a+1})}(i_{a},i_{a+1})H_{i,l(i_{a},i_{a+1})}\neq \emptyset.$$
Note that 
$\lambda_{l(i_{a},i_{a+1})}(i_{a},i_{a+1})A_{U_{l(i_{a},i_{a+1})}}=A_{U_{l(i_{a},i_{a+1})}}$ 
and
$$\lambda_{l(i_{a},i_{a+1})}(i_{a},i_{a+1})A_{U_{l(i_{a},i_{a+1})}\setminus\{v_{i,l(i_{a},i_{a+1})}\}}=\tau_{i_a,i_{a+1}} A_{U_{l(i_{a},i_{a+1})}\setminus\{v_{i,l(i_{a},i_{a+1})}\}}.$$
Because $H_{i, l(i_{a},i_{a+1})}$ and $\lambda_{l(i_{a},i_{a+1})} H_{i,l(i_{a},i_{a+1})}$ 
are of $v_{i,l(i_{a},i_{a+1})}$-type, we see that $H_{i,l(i_{a},i_{a+1})}= \lambda_{l(i_{a},i_{a+1})} H_{i,\lambda_{l(i_{a},i_{a+1})}}$. 
Then, we have 
\begin{equation*}
\begin{split}
[\lambda_{l(i_{a},i_{a+1})} A_{U_{l(i_{a},i_{a+1})}\setminus\{v_{i,l(i_{a},i_{a+1})}\}},\lambda_{l(i_{a},i_{a+1})} A_{U_{l(i_{a},i_{a+1})}}]\\
=[\tau_{i_a,i_{a+1}} A_{U_{l(i_{a},i_{a+1})}\setminus\{v_{i,{l(i_{a},i_{a+1})}}\}},A_{U_{l(i_{a},i_{a+1})}}] \subset N(H_{i,{l(i_{a},i_{a+1})}}).
\end{split}
\end{equation*}
Hence, we have 
$$\tau_{i_a,i_{a+1}}A_{U_{l(i_a,i_{a+1})}\setminus\{v_{i,l(i_a,i_{a+1})}\}}=A_{U_{l(i_a,i_{a+1})}\setminus\{v_{i,l(i_a,i_{a+1})}\}}.$$
This contradicts Lemma \ref{twist}.

Parts (4), (5), and (6) follow from (1), (2), and (3), respectively.
\end{proof}

The following is a key lemma.
\begin{lemma}\label{key}
For $a\in\{1,\ldots,r\}$, we have 
$$H_{i_a,l(i_a,i_{a+1})}\cap \lambda_{l(i_a,i_{a+1})}(i_a,i_{a+1})H_{i_{a+1},l(i_a,i_{a+1})}=\emptyset.$$
\end{lemma}
\begin{proof}
$H_{i_a,l(i_a,i_{a+1})}$ and $\lambda_{l(i_a,i_{a+1})}(i_a,i_{a+1})H_{i_{a+1},l(i_a,i_{a+1})}$ 
are of $v_{i_a,l(i_a,i_{a+1})}$-type and $v_{i_{a+1},l(i_a,i_{a+1})}$-type, respectively. 
Note that $v_{i_a,l(i_a,i_{a+1})}\neq v_{i_{a+1},l(i_a,i_{a+1})}$ implies 
$H_{i_a,l(i_a,i_{a+1})}\neq \lambda_{l(i_a,i_{a+1})}(i_a,i_{a+1})H_{i_{a+1},l(i_a,i_{a+1})}$. 
Additionally, $$\lambda_{l(i_a,i_{a+1})}(i_a,i_{a+1})A_{U_{l(i_a,i_{a+1})}}=A_{U_{l(i_a,i_{a+1})}} \text{ and}$$
$$\lambda_{l(i_a,i_{a+1})}(i_a,i_{a+1})A_{U_{l(i_a,i_{a+1})}\setminus\{v_{i_{a+1},l(i_a,i_{a+1})}\}}
=\tau_{l(i_a,i_{a+1})}A_{U_{l(i_a,i_{a+1})}\setminus\{v_{i_{a+1},l(i_a,i_{a+1})}\}}.$$ 
Consider 
$$[A_{U_{l(i_a,i_{a+1})}\setminus\{v_{i_a,l(i_a,i_{a+1})}\}}, A_{U_{l(i_a,i_{a+1})}}]
\subset N(H_{i_a,l(i_a,i_{a+1})}) \text{ and}$$
$$[\tau_{i_a,i_{a+1}}(i_a,i_{a+1})A_{U_{l(i_a,i_{a+1})}\setminus\{v_{i_{a+1},l(i_a,i_{a+1})}\}}, A_{U_{l(i_a,i_{a+1})}}]
\subset \lambda_{l(i_a,i_{a+1})}(i_a,i_{a+1})N(H_{i_{a+1},l(i_a,i_{a+1})}).$$
Now, assume that 
$$H_{i_a,l(i_a,i_{a+1})}\cap \lambda_{l(i_a,i_{a+1})}(i_a,i_{a+1})H_{i_{a+1},l(i_a,i_{a+1})}\neq \emptyset.$$
Because $$A_{U_{l(i_a,i_{a+1})}}\in N(H_{i_a,l(i_a,i_{a+1})})\cap \lambda_{l(i_a,i_{a+1})}(i_a,i_{a+1})N(H_{i_{a+1},l(i_a,i_{a+1})}),$$ 
the two edges $$[A_{U_{l(i_a,i_{a+1})}\setminus\{v_{i_a,l(i_a,i_{a+1})}\}}, A_{U_{l(i_a,i_{a+1})}}] \text{ and}$$ 
$$[\tau_{i_a,i_{a+1}}(i_a,i_{a+1})A_{U_{l(i_a,i_{a+1})}\setminus\{v_{i_{a+1},l(i_a,i_{a+1})}\}}, A_{U_{l(i_a,i_{a+1})}}]$$
must span a square. 
This contradicts Lemma \ref{twist}.
\end{proof}

Noting Lemma \ref{key0}, for any $i\in\{1,\ldots,k\}$, we define a sequence of hyperplanes 
$$\ldots,J_{i,-1}, J_{i,0}, J_{i,1},\ldots,J_{i,2rn}, J_{i,2rn+1}, \ldots,$$ 
a sequence of $k$-dimensional cubes 
$$\ldots,K_{-1},K_0,K_1,\cdots,K_{2rn},K_{2rn+1},\ldots,$$
and a sequence of vertices of $C_\Gamma$ 
$$\ldots,w_{-1},w_0,w_1,\ldots,w_{2rn},w_{2rn+1},\ldots$$ as follows. 
First, for $a\in\{1,\ldots,r\}$ and $l\in\{1,\dots,n\}$, we define 
\begin{align*}
J_{i,2((a-1)n+l)-1}&:=\gamma((a-1)n+l-1)H_{i,l},\\
J_{i,2((a-1)n+l)}&:=\gamma((a-1)n+l)H_{i,l},\\
K_{2((a-1)n+l)-1}&:=\gamma((a-1)n+l-1)[A_\emptyset,A_{U_l}],\\
K_{2((a-1)n+l)}&:=\gamma((a-1)n+l)[A_\emptyset,A_{U_l}],\\
w_{2((a-1)n+l)-1}&:=\gamma((a-1)n+l-1)A_{U_l},\\
w_{2((a-1)n+l)}&:=\gamma((a-1)n+l)A_\emptyset,\\
\end{align*}
where we note that both $J_{i,2((a-1)n+l)-1}$ and $J_{i,2((a-1)n+l)}$ are of the same $v_{i,l}$-type.

Second, for any $b\in\{1,\dots,2rn\}$ and $c\in \mathbb{Z}$, we set 
$$J_{i,2rnc+b}:=\gamma^cJ_{i,b}, K_{2rnc+b}:=\gamma^cK_{b}, w_{2rnc+b}:=\gamma^cw_{b}.$$
Then, we have $$K_{2rnc+b}=[w_{2rnc+b-1}, w_{2rnc+b}].$$
Additionally, we have the two connected components 
$J_{i,2rnc+b}^-$ and $J_{i,2rnc+b}^+$ such that 
$$C_\Gamma \!\setminus\!\!\setminus J_{i,2rnc+b}=J_{i,2rnc+b}^- \sqcup J_{i,2rnc+b}^+,$$ 
$$w_{2rnc+b-1}\in J_{i,2rnc+b}^-, w_{2rnc+b}\in J_{i,2rnc+b}^+.$$
Then, Lemma \ref{key0} implies that, for any $i\in\{1,\ldots,k\}$, 
$$\cdots\subsetneq J_{i,-1}^-\subsetneq  J_{i,0}^-\subsetneq J_{i,1}^-\subsetneq \cdots \subsetneq J_{i,2rn}^-\subsetneq J_{i,2rn+1}^-\subsetneq \cdots ,$$
$$\cdots\supsetneq J_{i,-1}^+\supsetneq J_{i,0}^+\supsetneq J_{i,1}^+\supsetneq \cdots \supsetneq J_{i,2rn}^+\supsetneq J_{i,2rn+1}^+\supsetneq \cdots.$$
Note that 
$$J_{i,0}^-\not\ni w_0=A_{\emptyset}\in J_{i,1}^-,$$
$$J_{i,2rn}^+\ni w_{2rn}=\gamma A_\emptyset \not\in J_{i,2rn+1}^+.$$

Let $\ell$ be a path from  $w_0=A_\emptyset$ to $w_{2rn}=\gamma A_\emptyset$ 
that diagonally penetrates each of the cubes $K_{1},\ldots,K_{2rn}$ in order. 
Then, the set of all hyperplanes intersecting the path $\ell$ is 
$\{J_{i,d}\}_{i\in\{1,\ldots,k\}, d\in\{1,\ldots,2rn\}}$. 
Hence, we have the following. 
\begin{lemma}\label{key3}
The set 
$\{J_{i,d}\}_{i\in\{1,\ldots,k\}, d\in\{1,\ldots,2rn\}}$
is the set of all hyperplanes separating $w_0=A_\emptyset$ and $w_{2rn}=\gamma A_\emptyset$. 
\end{lemma}

We now state the final lemma required for the proof of Theorem \ref{main}, 
where we recall Definition \ref{separatingdef} and Remark \ref{separatingrem}. 
\begin{lemma}\label{key4}
\begin{enumerate}
\item For any $i\in\{1,\ldots,k\}$, 
the sequence of hyperplanes $$J_{i,1},\ldots,J_{i,2rn}$$ 
is a sequence of separating hyperplanes 
from $J_{i,0}$ to $J_{i,2rn+1}$, where $J_{i,0}$ and $J_{i,2rn+1}$ 
are of $v_{i,n}$-type and $v_{i,1}$-type, respectively 
(in particular, from $w_0=A_\emptyset$ to $w_{2rn}=\gamma A_\emptyset$). 
\item 
The sequence of hyperplanes 
\begin{align*}
&J_{i_1,1},J_{i_1,2},\ldots,J_{i_1,2l(i_1,i_{2})-1},\\
&J_{i_{2},2l(i_1,i_{2})},J_{i_{2},2l(i_1,i_{2})+1},
\ldots,J_{i_{2},2n},\\
&J_{i_{2},2n+1},J_{i_{2},2(n+1)},
\ldots,J_{i_{2},2(n+l(i_{2},i_{3}))-1},\\
&\cdots \\
&J_{i_{r},2((r-2)n+l(i_{r-1},i_{r}))},J_{i_{r},2((r-2)n+l(i_{r-1},i_{r}))+1},
\ldots,J_{i_{r},2(r-1)n},\\
&J_{i_{r},2(r-1)n+1},J_{i_{r},2((r-1)n+1)},
\ldots,J_{i_{r},2((r-1)n+l(i_r,i_{r+1}))-1},\\
&J_{i_{r+1},2((r-1)n+l(i_r,i_{r+1}))},J_{i_{r+1},2((r-1)n+l(i_r,i_{r+1}))+1},
\ldots,J_{i_{r+1},2rn}
\end{align*}
is a sequence of separating hyperplanes 
from $J_{1,0}$ to $J_{1,2rn+1}$, where $J_{1,0}$ and $J_{1,2rn+1}$ 
are of $v_{1,n}$-type and $v_{1,1}$-type, respectively 
(in particular, from $w_0=A_\emptyset$ to $w_{2rn}=\gamma A_\emptyset$). 

Moreover, the sequence contains a hyperplane of $v_i$-type for any $i\in\{1,\ldots,k\}$ and $v_i\in V(\Gamma_i)$. 
\end{enumerate}
\end{lemma}
\begin{proof}
(1) The assertion is clear from parts (1) and (2) of Lemma \ref{key0}. 

(2) 
Parts (1), (2), and (3) of Lemma \ref{key0} and Lemma \ref{key} imply that 
the sequence specified in the assertion is a sequence of separating hyperplanes 
from $J_{i_1,0}$ to $J_{i_{r+1},2rn+1}$. 
Note that $i_1=1$ and $i_{r+1}=1$ by definition. 

We show that the sequence contains 
a hyperplane of $v_i$-type for any $i\in\{1,\ldots,k\}$ and $v_i\in V(\Gamma_i)$. 
Recall that $(V_{i_1},\dots,V_{i_r},V_{i_{r+1}})$ (\ref{pathtree}) is a closed path on the tree $T$. 
Thus, if an edge $(V_i,V_j)\in E(T)$ is contained in the closed path, then so is the inverse edge $(V_j,V_i)$. 
In addition, because $T$ is a spanning tree of $Q(\Gamma)$ and 
the closed path $(V_{i_1},\dots,V_{i_r},V_{i_{r+1}})$ passes through every vertex at least once, 
any $i\in \{1,\ldots ,k\}$, $V_i$ is contained in the closed path as a vertex. 
Additionally, recall that $T$ has the root $V_1$ 
and that $i<j$ only if $V_i$ is not farther than $V_j$ from $V_1$ in $T$. 

Now, take any $i\in\{1,\ldots,k\}$. 
We consider the two cases of $i=1$ and $i\neq 1$. 

First, suppose that $i=1$. 
Note that $i_1=1$ and set $j=i_2$. 
Take $a\in \{1,2,\ldots r\}$ such that $a\neq 1$, $i_a=1$, and $i_{a-1}=j$. 
Then, the sequence in the assertion contains the two subsequences 
\begin{align*}
J_{i_1,1},J_{i_1,2},\ldots,J_{i_1,2l(i_1,i_{2})-1},
\end{align*}
which are of $v_{1,1}$-type, $v_{1,1}$-type, $\ldots$, 
$v_{1,l(1,j)}$-type, and 
\begin{align*}
J_{i_{a},2((a-2)n+l(i_{a-1},i_{a}))},J_{i_{a},2((a-2)n+l(i_{a-1},i_{a}))+1},
\ldots,J_{i_{a},2(a-1)n},
\end{align*}
which are of $v_{1,l(j,1)}$-type, $v_{1,l(j,1)+1}$-type, $\ldots$, 
$v_{1,n}$-type. 
Note that $l(j,1)=l(1,j)$ by Lemma \ref{closed}. 
Take any vertex $v_1\in V(\Gamma_1)$. 
Then, there exists some $l\in\{1,\ldots,n\}$ such that $v_{1,l}=v_1$ by Lemma \ref{closed}. 
Hence, the sequence  in the assertion contains a hyperplane of $v_1$-type. 

Next, suppose that $i\neq 1$.  
Take the smallest $a\in \{1,\ldots r\}$ such that $i_a=i$. 
Because $i_1=1$ and $i\neq 1$, we have $a\neq 1$. 
We set $j=i_{a-1}$. 
Then, we have $a'\in \{1,\ldots r\}$ such that 
$a\le a'$, $i_{a'}=i$, and $i_{a'+1}=j$. 
The sequence contains the two subsequences 
\begin{align*}
J_{i_{a},2((a-2)n+l(i_{a-1},i_{a}))},J_{i_{a},2((a-2)n+l(i_{a-1},i_{a}))+1},
\ldots,J_{i_{a},2(a-1)n},
\end{align*}
which are of $v_{i,l(j,i)}$-type, $v_{i,l(j,i)+1}$-type, $\ldots$, 
$v_{i,n}$-type, and 
\begin{align*}
J_{i_{a'},2(a'-1)n+1},J_{i_{a'},2((a'-1)n+1)},
\ldots,J_{i_{a'},2((a'-1)n+l(i_{a'},i_{a'+1}))-1},
\end{align*}
which are of $v_{i,1}$-type, $v_{i,1}$-type, $\ldots$, 
$v_{i,l(j,i)}$-type. 
Note that $l(j,i)=l(i,j)$ by Lemma \ref{closed}. 
Take any vertex $v_i\in V(\Gamma_i)$. 
Then, there exists some $l\in\{1,\ldots,n\}$ such that $v_{i,l}=v_i$ by Lemma \ref{closed}. 
Hence, the sequence in the assertion contains a hyperplane of $v_i$-type. 
\end{proof}

\begin{figure}
\begin{center}
\includegraphics[width=12cm,pagebox=cropbox,clip]{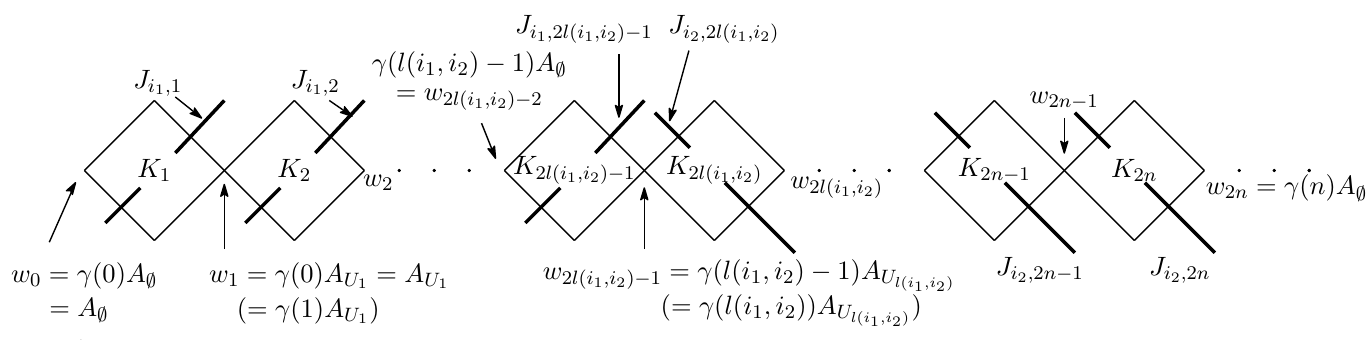}
\caption{Part of a sequence of hyperplanes 
for the case where $\Gamma=\Gamma_1\ast\Gamma_2$.}\label{sequence_fig}
\end{center}
\end{figure}

We can now complete the proof of Theorem \ref{main}. 
\begin{proof}[Proof of $(1)\Rightarrow(2)$ in Theorem \ref{main}]
We show $(1)\Rightarrow(2)$ in Theorem \ref{main}. 
Consider $\gamma\in A_\Gamma$ defined by (\ref{gammaelement}) 
and hyperplanes 
$$J:=J_{1,0} \text{ and } J':=J_{1,2rn+1},$$ 
which are of $v_{1,n}$-type and $v_{1,1}$-type, respectively. 
We will confirm conditions (i), (ii), and (iii) in Theorem \ref{criterion}. 

(i) $\gamma$ skewers $(J,J')$. Indeed, it is clear that 
$$J^+\supsetneq (\gamma^{-1}(J'^+)\supsetneq\gamma(J^+)\supsetneq) J'^+\supsetneq \gamma^2(J^+).$$

(ii) We show that $J$ and $J'$ are strongly separated. 
Take any hyperplane $H$ with $J\cap H\neq \emptyset$. 
When $H$ is of $v_i$-type for some $i$ and $v_i\in V(\Gamma_i)$, 
take a hyperplane $H'$ of $v_i$-type 
separating $A_\emptyset$ and $\gamma A_\emptyset$ 
such that $J\cap H'=\emptyset$ by part (2) of Lemma \ref{key4}. 
Then, $H\cap H'=\emptyset$ by Remark \ref{label}. 
Because $J\subset H'^-, H'^+\supset J'$,  $J\cap H\neq \emptyset$ and $H\cap H'=\emptyset$, we have 
$J'\cap H=\emptyset$. 

(iii) We show $Stab(J)\cap Stab(J')=\{1\}$. 
Note that for any $i\in\{1,\ldots,k\}$ and any $v_i\in V(\Gamma_i)$, 
we have at least one sequence of separating hyperplanes 
$P'_1,\ldots,P'_{M'}$ from $A_\emptyset$ to $\gamma A_\emptyset$ 
such that $P'_1$ is of $v_i$-type and $P'_{M'}$ is of $v_{1,n}$-type. 
For example, we can take such a sequence 
by considering a subsequence of 
the sequence in part (2) of Lemma \ref{key4}. 
For any $i\in\{1,\ldots,k\}$ and any $v_i\in V(\Gamma_i)$, 
we can take a longest sequence of separating hyperplanes 
$P_1,\ldots,P_M$ from $A_\emptyset$ to $\gamma A_\emptyset$ 
such that $P_1$ is of $v_i$-type and $P_M$ is of $v_{1,n}$-type, 
where 
$$A_\emptyset \in P_1^-, P_1^+\supsetneq P_2^+\supsetneq\cdots 
\supsetneq P_{M-1}^+\supsetneq P_M^+\ni \gamma A_\emptyset$$ 
by taking decompositions by appropriate connected components 
$$C_\Gamma\!\setminus\!\!\setminus P_1
=P_1^-\sqcup P_1^+,\ldots,C_\Gamma\!\setminus\!\!\setminus P_M=P_M^-\sqcup P_M^+.$$
Note that $P_1,\ldots,P_M\in \{J_{j,d}\}_{j\in\{1,\ldots,k\}, d\in\{1,\ldots,2rn\}}$ 
by Lemma \ref{key3}. 
By noting $P_M\in \{J_{1,d}\}_{d\in\{1,\ldots,2rn\}}$ 
and part (1) of Lemma \ref{key4} for the case $i=1$, 
we have $P_M^+ \supsetneq J'^+$. 

Now assume that $Stab(J)\cap Stab(J')\neq \{1\}$. 
Take $g\in Stab(J)\cap Stab(J')$ with $g\neq 1$.
Note that $g^{-1}\in Stab(J)\cap Stab(J')$. 
Then, we have a hyperplane $H$ of $v_i$-type for some $i\in\{1,\ldots,k\}$ 
and some $v_i\in V(\Gamma_i)$, 
separating $A_\emptyset$ and $gA_\emptyset$. 
We take a longest sequence of separating hyperplanes 
$P_1,\ldots,P_M$ from $A_\emptyset$ to $\gamma A_\emptyset$ 
such that $P_1$ is of $v_i$-type and $P_M$ is of $v_{1,n}$-type. 
Then we have 
$$g\gamma A_\emptyset, g^{-1}\gamma A_\emptyset\in P_M^+$$
by $P_M^+ \supsetneq J'^+$. 
We take a connected component 
$H^-$ of $C_\Gamma\!\setminus\!\!\setminus H$ such that $A_\emptyset\in H^-$. 
Then the other connected component $H^+$ satisfies $gA_\emptyset\in H^+$. 
Because $H\cap J\neq \emptyset$, we have $H\cap J'=\emptyset$ by (ii). 
Thus $H$ can not separate 
$\gamma A_\emptyset$ and $g\gamma A_\emptyset$. 
Hence we have either 
\begin{enumerate}
\item[(a)] 
$A_\emptyset, \gamma A_\emptyset,g \gamma A_\emptyset \in H^-  \text{ and } gA_\emptyset \in H^+$ or 
\item[(b)]
$A_\emptyset\in H^- \text{ and } gA_\emptyset, \gamma A_\emptyset,g \gamma A_\emptyset \in H^+$.
\end{enumerate}

Assume that the case (a) occurs. Then $H$ does not separate 
$A_\emptyset$ and $\gamma A_\emptyset$ and thus 
$H\notin\{J_{i,d}\}_{d\in\{1,\ldots,2rn\}}$ by Lemma \ref{key3}. 
Hence $H$ and $P_1$ are different hyperplanes of the same $v_i$-type, and thus 
they can not intersect by Remark \ref{label}. 
Then we have either $H^-\supsetneq P_1^+$ or $H^+\supsetneq P_1^+$. 
However $H^+\supsetneq P_1^+$ can not occur because 
$\gamma A_\emptyset \notin H^+$ and $\gamma A_\emptyset \in P_1^+$. 
Therefore 
$$H^-\supsetneq P_1^+.$$ 
Then we have 
$$gA_\emptyset \in H^+, H^-\supsetneq P_1^+\supsetneq P_2^+\supsetneq\cdots 
\supsetneq P_{M-1}^+\supsetneq P_M^+\ni g\gamma A_\emptyset.$$ 
Then, $Q_1=H, Q_2=P_1,\ldots,Q_{M+1}=P_M$ is a sequence of separating hyperplanes 
from $gA_\emptyset$ to $g\gamma A_\emptyset$ 
such that $Q_1$ is of $v_i$-type and $Q_{M+1}$ is of $v_{1,n}$-type. 
Thus, $g^{-1}Q_1, g^{-1}Q_2,\ldots,g^{-1}Q_{M+1}$ is a sequence of separating hyperplanes 
from $A_\emptyset$ to $\gamma A_\emptyset$ 
such that $g^{-1}Q_1$ is of $v_i$-type and $g^{-1}Q_{M+1}$ is of $v_{1,n}$-type, 
which contradicts the fact that the sequence $P_1,\ldots,P_M$ is longest. 

Next, assume that the case (b) occurs. Then $H$ does not separate 
$gA_\emptyset$ and $g\gamma A_\emptyset$, that is, 
$g^{-1}H$ does not separate 
$A_\emptyset$ and $\gamma A_\emptyset$ and thus 
$g^{-1}H\notin\{J_{i,d}\}_{d\in\{1,\ldots,2rn\}}$ by Lemma \ref{key3}. 
Hence $g^{-1}H$ and $P_1$ are different hyperplanes of the same $v_i$-type, and thus 
they can not intersect by Remark \ref{label}. 
Then we have either $g^{-1}H^+\supsetneq P_1^+$ or $g^{-1}H^-\supsetneq P_1^+$. 
However $g^{-1}H^-\supsetneq P_1^+$ can not occur because 
$\gamma A_\emptyset \notin g^{-1}H^-$ and $\gamma A_\emptyset \in P_1^+$. 
Therefore 
$$g^{-1}H^+\supsetneq P_1^+.$$ 
Then we have 
$$g^{-1}A_\emptyset \in g^{-1}H^-, g^{-1}H^+\supsetneq P_1^+\supsetneq P_2^+\supsetneq\cdots 
\supsetneq P_{M-1}^+\supsetneq P_M^+\ni g^{-1}\gamma A_\emptyset.$$ 
Then, $Q_1=H, Q_2=gP_1,\ldots,Q_{M+1}=gP_M$ is a sequence of separating hyperplanes 
from $A_\emptyset$ to $\gamma A_\emptyset$ 
such that $Q_1$ is of $v_i$-type and $Q_{M+1}$ is of $v_{1,n}$-type, 
which contradicts the fact that the sequence $P_1,\ldots,P_M$ is longest. 

We now see that $Stab(J)\cap Stab(J')=\{1\}$.

\end{proof}

\begin{remark}\label{mainremark}
In \cite[Theorem 3.3]{Charney}, 
it is shown that $A_\Gamma$ is centerless 
under the setting in Theorem \ref{main} with (1). 
This claim can be proved based on Theorem \ref{main}. 
Indeed, Theorem \ref{main} $(1)\Rightarrow(2)$ implies that $A_\Gamma$ is acylindrically hyperbolic. 
Therefore, the center of $A_\Gamma$ is finite by acylindrical hyperbolicity (\cite[Corollary 7.3]{Osin}). 
Proposition \ref{finite normal subgroup} then implies that the center is trivial. 
Hence, $A_\Gamma$ is centerless. 
\end{remark}


\section*{Acknowledgements}
The authors would like to thank Sam Shepherd for helpful comments. The authors would like to thank the anonymous reviewers for comments on the previous version of this paper.


\end{document}